\documentclass[12pt]{article}

\usepackage[margin=1in]{geometry}

\usepackage{amsmath}
\usepackage{graphicx}
\usepackage{enumerate}
\usepackage[round]{natbib}
\usepackage{url} 

\usepackage{amsfonts}
\usepackage{amssymb}
\usepackage{amstext,amsthm}
\usepackage{mathtools}
\usepackage{xr}
\usepackage{subcaption}
\usepackage{float}
\usepackage{setspace}
\usepackage{color}

\newtheorem{comments}{Comment}
\newcommand{\ncalo}{|{\cal O}|}
\newcommand{\fncalo}{\frac{1}{|{\cal O}|}}
\newcommand{\ty}{{\mathbf{Y}_0}}
\newcommand{\hz}{{\rm hr}}
\newcommand{\ddd}{{\rm did}}
\newcommand{\vt}{{\rm vt}}

\newcommand{\divv}{{{\rm vt-en}}}
\newcommand{\scc}{{\rm sc-adh}}
\newcommand{\oper}{{\rm op}}

\newcommand{\tba}{\tilde{\ba}}
\newcommand{\bae}{\mathbf{a}_1}
\newcommand{\bat}{\mathbf{a}_2}
\newcommand{\batt}{\mathbf{a}_2^\top}
\newcommand{\batth}{\mathbf{a}_3}
\newcommand{\mc}{{\rm mc-nnm}}
\newcommand{\bye}{\mathbf{y}_1}
\newcommand{\byt}{\mathbf{y}_2}
\newcommand{\bytt}{\mathbf{y}_2^\top}
\newcommand{\hr}{{\rm hr}}

\newcommand{\shrink}{{\rm shrink}}

\newcommand{\tick}{\checkmark}

\newcommand{\mmax}{\textrm{max}}
\newtheorem{assumption}{Assumption}
\newtheorem{theorem}{Theorem}
\newtheorem{lemma}{Lemma}
\newtheorem{prop}{Proposition}

\def\monthname{\ifcase\month\or
	January\or February\or March\or April\or May\or June\or July\or
	August\or September\or October\or November\or December\fi}
\numberwithin{equation}{section}

\newcommand{\bbb}{{\bf b}}

\newcommand{\ba}{{\bf A}}
\newcommand{\bh}{{\bf H}}
\newcommand{\be}{{\boldsymbol{\frak E}}}
\newcommand{\bl}{{\bf L}}

\newcommand{\by}{{\bf Y}}
\newcommand{\bx}{{\bf X}}
\newcommand{\bz}{{\bf Z}}
\newcommand{\bs}{{\bf S}}
\newcommand{\br}{{\bf R}}
\newcommand{\bu}{{\bf U}}
\newcommand{\bv}{{\bf V}}
\newcommand{\bw}{{\bf W}}

\newcommand{\calo}{{\cal O}}
\newcommand{\mmm}{\mathbb{M}}
\newcommand{\calm}{{\cal M}}
\newcommand{\reals}{{\mathbb{R}}}
\newcommand{\id}{{\mathbf{1}}}

\newcommand{\tbH}{\tilde{\bh}}
\newcommand{\tbx}{\tilde{\bx}}
\newcommand{\tbz}{\tilde{\bz}}

\newcommand{\yitn}{Y_{it}(0)}
\newcommand{\yite}{Y_{it}(1)}
\newcommand{\yit}{Y_{it}}
\newcommand{\wit}{W_{it}}

\newcommand{\err}{\varepsilon}

\newcommand{\tr}{\mathrm{trace}}
\newcommand{\op}{\mathrm{op}}

\newcommand{\prob}{\mathbb{P}}
\newcommand{\bbeta}{\boldsymbol{\eta}}

\newcommand{\bepsilon}{{\boldsymbol{\varepsilon}}}
\newcommand{\bzero}{{\bf 0}}
\newcommand{\bSigma}{{\bf \Sigma}}
\newcommand{\bb}{{\bf B}}
\newcommand{\E}{\mathbb{E}}
\newcommand{\ind}{\mathbb{I}}
\newcommand{\bdelta}{\mathbf{\Delta}}
\newcommand{\rank}{\text{rank}}
\newcommand{\calc}{\mathcal{C}}
\newcommand{\lmax}{L_{\max}}

\newcommand{\pc}{p_{\text{c}}}
\newcommand{\nc}{N_{\text{c}}}

\newcommand{\bi}{{\bf I}}
\newcommand{\bp}{{\bf P}}
\newcommand{\bOmega}{{\bf \Omega}}
\newcommand{\ce}{e}
\newcommand{\bce}{{\bf E}}
\newcommand{\cX}{\mathcal{X}}
\newcommand{\ltpi}{{L^2(\Pi)}}
\newcommand{\cC}{\mathcal{C}}
\newcommand{\cB}{\mathcal{B}}

\newcommand{\bM}{{\bf M}}
\newcommand{\tZ}{\widetilde{Z}}

\newcommand{\update}{}

\newtheorem{remark}{Remark}[section]

\DeclareMathOperator*{\argmin}{arg\,min}

\begin{document}

		\title{\bf Matrix Completion Methods for Causal Panel Data Models
			\thanks{We are
				grateful for comments by Alberto Abadie{\update, Ran Chen, }and participants at the NBER Summer Institute and at seminars at Stockholm University and the 2017 California Econometrics Conference. This research was generously supported by ONR grant N00014-17-1-2131 and NSF grant CMMI:1554140.}}
		\author{Susan Athey \thanks{{\small Professor of Economics,  Graduate School of Business, Stanford University, SIEPR, and NBER, athey@stanford.edu. }} \and Mohsen Bayati\thanks{{\small Associate Professor, Graduate School of Business, Stanford University,
					bayati@stanford.edu.}}
			\and Nikolay Doudchenko\thanks{{\small Google Research, 111 8th Ave, New York, NY 10011, nikolayd@google.com.}}
			\and Guido Imbens\thanks{{\small Professor of
					Economics,
					Graduate School of Business, and Department of Economics, Stanford University, SIEPR, and NBER,
					imbens@stanford.edu.}}
			\and Khashayar Khosravi\thanks{{\small
					Department of Electrical Engineering, Stanford University,
					khosravi@stanford.edu.}}
		}
		\date{}
		\maketitle

	\begin{abstract}
	
		In this paper we study methods for estimating causal effects in settings with panel data, where  some units  are exposed to a treatment during some periods and the goal is estimating counterfactual (untreated) outcomes
		for the treated unit/period combinations.  We propose a class of matrix completion estimators that uses the observed elements of the matrix of
		control outcomes corresponding to untreated unit/periods to impute the ``missing''
		elements of the control outcome matrix, corresponding to treated units/periods.  This leads to a matrix that well-approximates the
		original (incomplete) matrix, but has lower complexity according to the nuclear norm for matrices.
We generalize results from the matrix completion literature by allowing the patterns of missing data to have a time series dependency structure that is common in social science applications.
		We present novel insights concerning the connections between the matrix completion literature, the literature on interactive fixed effects models and the literatures on program evaluation under unconfoundedness and synthetic control methods. We show that all these estimators can be viewed as focusing on the same objective function. They differ solely in the way they deal with  identification, in some cases solely through regularization (our proposed nuclear norm matrix completion estimator) and in other cases primarily through imposing hard restrictions (the unconfoundedness and synthetic control approaches). The proposed method outperforms unconfoundedness-based or synthetic control estimators in simulations based on real data.
	
	\end{abstract}
	
	\noindent%

	\newpage
	
	\section{Introduction}

	In this paper we develop new methods for estimating average causal effects in settings with panel or longitudinal data, where some units  are exposed to a binary treatment during some periods.
	To estimate the average causal effect of the treatment on the treated units in this setting, we impute the missing potential control outcomes.
	
	The statistics and econometrics causal inference literatures have taken two general approaches to this problem. The literature on unconfoundedness (\citet{rosenbaum1983central, imbens2015causal}) can be interpreted as imputing missing potential control outcomes for treated units using observed control outcomes for control units with similar values for observed outcomes in previous periods. In contrast, the recent synthetic control literature (\citet{abadie2003, abadie2010synthetic, abadie2014, doudchenko, ben2018augmented,
		li2019statistical, ferman2019synthetic, arkhangelsky2019synthetic,  chernozhukov2017exact}, see \citet{abadie2019using} for a review) imputes missing control outcomes for treated units using weighted average outcomes for control units with the weights chosen so that the weighted lagged control outcomes match the lagged outcomes for treated units. Although at first sight similar, the two approaches are conceptually quite different in terms of the correlation patterns in the data they exploit to impute the missing potential outcomes.
	The unconfoundedness approach assumes that patterns over time  are stable across units, and the synthetic control approach assumes that patterns across units are stable over time.
	In empirical work the two  sets of methods   have primarily been applied in settings with  different structures on the missing data or  assignment mechanism. In the case of the unconfoundedness literature  the typical setting is one with the treated units all treated in the same periods, typically only the last period, and  with a substantial number of control and treated units. The synthetic control literature has primarily focused on the setting with one or a small number of treated units observed prior to the treatment over a substantial number of periods. We argue that once regularization methods are used, the two approaches, unconfoundedness and synthetic controls, are applicable in the same settings, leaving the researcher with a real choice in terms of methods. In addition this insight allows for a more systematic comparison of their perfomance than has been appreciated in the literature.
	
	In this study we  draw on the econometric literature on factor models and interactive fixed effects, and the computer science and statistics literatures on matrix completion, to take an approach to imputing the missing potential outcomes that is different from the unconfoundedness and synthetic control approaches. In fact, we show that it can  be viewed as nesting both. In the literature on factor models and interactive effects (\citet{bai2002determining, bai2003inferential})  researchers model the observed outcome as the sum of a linear function of covariates and an unobserved component that is a low rank matrix plus noise. Estimates are typically based on minimizing the sum of squared errors given the rank of the matrix of unobserved components, sometimes with the rank estimated.
	\citet{xu2017generalized} extends these ideas to causal settings where a  subset of units is treated from a common period onward, so that  complete data methods for estimating the factors and factor loadings can be exploited.
	The matrix completion literature (\citet{candes2009exact, candes2010matrix, mazumder2010spectral}) focuses on imputing missing elements in a matrix assuming that: $(i)$ the complete matrix is the sum of a low rank matrix plus noise and $(ii)$, the missingness is completely at random (except \cite{gamarnik2016note} that study a stylized rank one case). The rank of the matrix is implicitly determined by the regularization through the addition of a penalty term to the objective function. Especially with complex missing data patterns using the nuclear norm as the regularizer is attractive for computational reasons.
	
	In the current paper we make three contributions.
	First, we present formal results for settings where the missing data patterns are not completely at random and have a structure that allows for correlation over time, generalizing the results from the matrix completion literature. In particular we allow for the possibility of staggered adoption ({\it e.g.,} \citet{athey2018design, shaikh2019randomization}), where units are treated from some initial adoption date onwards, but the adoption dates vary between units.
	We also modify the estimators from the matrix completion and factor model literatures to allow for unregularized unit and time fixed effects. Although these can be incorporated in the low rank matrix, in practice the performance of the estimator with the unregularized two-way fixed effects is substantially better. Compared to the factor model literature in econometrics the proposed estimator focuses on nuclear norm regularization to avoid the computational difficulties  that would arise for complex missing data patterns with the fixed-rank methods in
	\citet{bai2002determining} and \citet{xu2017generalized}, similar to the way LASSO (or $\ell_1$ regularization, \citet{tibshirani1996regression}) is computationally attractive relative to subset selection (or $\ell_0$ regularization) in linear regression models.
	The second contribution is to show that the synthetic control and unconfoundedness approaches, as well as our proposed method, can all be viewed as matrix completion methods based on matrix factorization, all with the same objective function based on the Fr\"obenius norm for the difference between the latent matrix and the observed matrix.
	Given this common objective function, the unconfoundedness and synthetic control approaches impose different sets of restrictions on the factors in the matrix factorization. In contrast, the proposed method does not impose any restrictions but uses regularization to characterize the estimator.
	In our third contribution we apply our methods to two real data sets where we observe the complete matrix. We artificially designate outcomes for some units and time periods to be missing, and then compare the performance of different imputation estimators. We find that the nuclear norm matrix completion estimator does well in a range of cases, including when $T$ is small relative to $N$, when $T$ is large relative to $N$, and when $T$ and $N$ are comparable. In contrast, the unconfoundedness and synthetic control approaches break down in some of these settings in the expected pattern (the unconfoundedness approach does not work very well if $T\gg N$, and the synthetic control approach does not work very well if $N\gg T$).
	
	We discuss some extensions in the final part of the paper. In particular we consider extensions to settings where the probability of assignment to the treatment may vary systematically with observed characteristics. In the program evaluation literature such settings have often been addressed using inverse propensity score weighting (\citet{rubin2006matched, hirano2003efficient}), which can be applied here as well.
	
	%
	
	\section{Set Up}\label{section_setup}
	
	We start by stating the causal problem.
	Consider a setting with $N$ units observed over $T$ periods. In each period each unit is characterized by two potential outcomes, $Y_{it}(0)$ and $Y_{it}(1)$. In period $t$ unit $i$ is exposed or not to a binary treatment, with $W_{it}=1$ indicating that the unit is exposed to the treatment and $W_{it}=0$ otherwise. We observe for each unit and period the pair $(W_{it},Y_{it})$ where the realized outcome is $\yit=\yit(\wit)$.
	In addition to  observing the matrix $\by$ of realized outcomes and the matrix of treatment assignments $\bw$, we
	may also observe covariate matrices $\bx\in\reals^{N\times P}$ and $\bz\in\reals^{T\times Q}$ where columns of $\bx$ are
	unit-specific covariates, and columns of $\bz$ are time-specific covariates. We may also observe unit/time specific covariates $V_{it}\in\reals^J$.
	Implicit in our set up is that we rule out dynamic effects and make the stable-unit-treatment-value assumption (\citet{rubin2006matched, imbens2015causal}): the potential outcomes are indexed only by the contemporaneous treatment for that unit and not by past treatments or treatments for other units. Cases where such assumptions are restrictive include those analyzed in the dynamic treatment regime literature (\citet{chamberlain1993feedback, hernan2010causal}).
	In the case where units are only exposed to the treatment in the last period this issue is not material. Also, in the case with staggered adoption violations of the no-dynamics assumption simply changes the interpretation of the estimand, but does not in general invalidate a causal interpretation.
	
	Here we focus  on estimating the average effect for the treated, $\tau=\sum_{(i,t):W_{it}=1} [\yite-\yitn]/\sum_{i,t} W_{it}$, although other averages such as the overall average causal effect, $\sum_{i,t} [\yite-\yitn]/(NT)$, could be of interest too.
	In order to estimate such average treatment effects, one approach is to impute the missing potential outcomes.  Because we focus  on estimating the average effect for the treated, all the relevant values for $Y_{it}(1)$ are observed, and thus we only need to impute the missing entries in the $\by(0)$ matrix for treated units with $W_{it}=1$. For the moment we focus on the problem of imputing the missing entries in $\by(0)$ given the observed values of  ${\bf Y}(0)$ and the observed matrix $\bw$. To ease the notation and facilitate the connection to the matrix completion literature we drop from here on the $(0)$ part of $\by(0)$ and simply focus on imputing the missing values of  a partially observed matrix $\by$ (with the understanding that this would be the matrix of control outcomes $\by(0)$), with $\bw$ the matrix of missing data (treatment assignment) indicators.
	One may also wish to use the observed values of $\by(1)$ for imputing the missing values for $\by(0)$, but we do not do so here. In setting with few values of $\by(1)$ observed it is unlikely that the information in these values is important. (In particular in the case we focus on for part of this study, with only a single treated unit/period pair there would be no information in this value.).  Extension to the cases that leverage also data from $\by(1)$ require assumptions on the treatment effect and are briefly discussed in \S \ref{subsec:impute-y0-y1}.

	For any positive integer $n$, we use notation $[n]$ to refer to the set $\{1,\ldots,n\}$ and use $\id_n$ to denote the $n$ by $1$ vector of all ones.
	We define $\calm$ to be the set of pairs of indices $(i,t)$, $i\in[N]$, $t\in[T]$, corresponding to the  missing entries with $W_{it}=1$ and $\calo$ to be the set os pairs of indices corresponding to  the observed entries in $\by$ with $W_{it}=0$.
	Putting aside the covariates for the time being, the data can be thought of as consisting  of two $N\times T$ matrices, one incomplete and one complete,
	{\small
		\begin{equation}\label{een} \by=\left(
			\begin{array}{ccccccc}
				Y_{11} & Y_{12} & ?  & \dots & Y_{1T}  \\
				?  & ? & Y_{23}   & \dots & ?  \\
				Y_{31}  & ? & Y_{33}   & \dots & ?  \\
				\vdots   &  \vdots & \vdots &\ddots &\vdots \\
				Y_{N1}  & ? & Y_{N3}   & \dots & ?  \\
			\end{array}
			\right),\hskip1cm
			{\rm and}\ \  \bw=\left(
			\begin{array}{ccccccc}
				0 & 0 & 1  & \dots & 0  \\
				1  & 1 & 0   & \dots & 1  \\
				0  & 1 & 0   & \dots & 1  \\
				\vdots   &  \vdots & \vdots &\ddots &\vdots \\
				0  & 1 & 0   & \dots & 1 \\
			\end{array}
			\right),\end{equation}
	}
	where
	\[W_{it}=\left\{\begin{array}{ll}
	1\hskip1cm & {\rm if}\ (i,t)\in\calm,\\
	0\hskip1cm & {\rm if}\ (i,t)\in\calo,\end{array}\right.\]
	is an indicator for the event that the corresponding component of $\by$, that is, $Y_{it}$, is missing.
	The main part of the paper is about the statistical problem of imputing the missing values in $\by$. Once these are imputed we can then estimate the average causal effect of interst, $\tau$.

	%
	
	\section{Patterns of Missing Data, Thin and Fat Matrices,  and Horizontal and Vertical Regression}
	
	In this section, we discuss a number of particular configurations of the matrices $\by$ and $\bw$ that are the focus of distinct parts of the general literature. This discussion serves to put in context the problem, and  to motivate previously developed methods from the  literature on causal inference under unconfoundedness, the synthetic control literature, and the interactive fixed effect literature, and subsequently to develop formal connections between all three and the matrix completion literature. Note that the matrix completion literature has focused primarily on the case where $\bw$ is completely random,
	as in Equation (\ref{een}), and where both dimensions of $\by$ and $\bw$ are large.
	First, we consider patterns of missing data, that is, distributions for $\bw$ that differ from completely random. Second, we consider different shapes of the matrix $\by$ where the relative size of the dimensions $N$ and $T$ may be very different and one or both may be modest in magnitude. Third, we consider a number of specific analyses in the econometrics literature that focus on particular combinations of missing data patterns and shapes of the matrices.
	
	\subsection{Patterns of Missing Data}
	
	In the statistics and computer science literatures on matrix completion the focus is typically on settings with randomly missing values, allowing for general patterns on the matrix of missing data indicators \citep{candes2010thepower,recht2011simpler}.
	In contrast in causal social science applications the missingness arises from treatment assignments and the choices that lead to these assignments. As a result  are often specific structures on the missing data that depart substantially from complete randomness.
	
	\subsubsection{Block Structure}
	A leading example is a block structure, with a subset of the units adopting an irreversible treatment at a particular point in time $T_0+1$. In the example below the $\tick$ marks indicate observed values and the $?$ indicate missing values.
	{\small
		\[ \by_{N\times T}=\left(
		\begin{array}{ccccccc}
		\tick & \tick & \tick & \tick  & \dots & \tick \\
		\tick & \tick  & \tick & \tick   & \dots & \tick  \\
		\tick & \tick & \tick & \tick   & \dots & \tick  \\
		\tick & \tick & \tick & ?   & \dots & ?  \\
		\tick & \tick & \tick & ?   & \dots & ?  \\
		\vdots   &  \vdots   & \vdots & \vdots &\ddots &\vdots \\
		\tick & \tick & \tick & ?   & \dots & ?  \\
		\end{array}
		\right)\,.
		\]
	}
	There are two special cases of the block structure.
	Much of the literature on estimating average treatment effects under unconfoundedness  ({\it e.g.,} \citet{imbens2015causal}) focuses on the case where $T_0=T-1$, so that the only treated units are in the last period. We will refer to this as the single-treated-period block structure.
	In contrast, the synthetic control literature ({\it e.g.}, \citet{abadie2010synthetic, abadie2019using})  focuses primarily on the case of  with a single treated unit which are treated for a number of periods from period $T_0+1$ onwards, the single-treated-unit block structure:
	{\small
		\[ \by=\left(
		\begin{array}{cccccr}
		\tick & \tick & \tick  & \dots & \tick& \tick\\
		\tick  & \tick & \tick   & \dots & \tick& \tick  \\
		\tick  & \tick & \tick   & \dots & \tick& ?  \\
		\vdots   &  \vdots & \vdots &\ddots & \vdots  &\vdots \\
		\tick  & \tick & \tick   & \dots & \tick& ?  \\
		&&&&& \uparrow
		\\
		\multicolumn{6}{r}{\rm treated\ period}\end{array}
		\right)\hskip0.5cm{\rm and}\ \
		\by=\left(
		\begin{array}{cccccr}
		\tick & \tick & \tick  & \dots & \tick \\
		\tick  & \tick & \tick   & \dots & \tick  \\
		\tick & \tick & \tick   & \dots & \tick  \\
		\vdots   &  \vdots & \vdots &\ddots &\vdots \\
		\tick & \tick & \tick   & \dots & \tick  \\
		\tick  & \tick & ?   & \dots & ?  & \leftarrow{\rm treated\ unit}\\
		\end{array}
		\right)\,.\]
	}
	A special case that fits in both these settings is that with a single missing unit/time pair:
	
	{\small
		\[ \by=\left(
		\begin{array}{cccccr}
		\tick & \tick & \tick  & \dots & \tick& \tick\\
		\tick  & \tick & \tick   & \dots & \tick& \tick  \\
		\tick  & \tick & \tick   & \dots & \tick& \tick  \\
		\vdots   &  \vdots & \vdots &\ddots & \vdots  &\vdots \\
		\tick  & \tick & \tick   & \dots & \tick& \tick  \\
		\tick  & \tick & \tick   & \dots & \tick& ?  \end{array}
		\right).\]
	}
	This specific setting is useful to contrast methods developed for the single-treated period (unconfoundedness) case with those developed for the single-treated unit (synthetic control) case because both sets of methods are potentially applicable.
	
	\subsubsection{Staggered Adoption}
	
	Another setting that has received attention is the  staggered adoption design (\citet{athey2018design, shaikh2019randomization}). Here units may differ in the time they  are first exposed to the treatment, but  the treatment is irreversible. This naturally arises in settings where the treatment is some new technology that units can choose to adopt (e.g., \citet{athey2002impact}). Here:
	{\small
		\[ \by_{N\times T}=\left(
		\begin{array}{ccccccr}
		\tick & \tick & \tick & \tick  & \dots & \tick & {\rm (never\ adopter)}\\
		\tick & \tick  & \tick & \tick   & \dots & ?  & {\rm (late\ adopter)}\\
		\tick & \tick &? & ?   & \dots & ?  \\
		\tick & \tick & ? & ?   & \dots & ? &\ \ \  {\rm (medium\ adopter)} \\
		\vdots   &  \vdots   & \vdots & \vdots &\ddots &\vdots \\
		\tick & ? & ? & ?   & \dots & ? & {\rm (early\ adopter)}  \\
		\end{array}
		\right)\,.
		\]
	}

	\subsection{Thin and Fat Matrices}

	A second classification of the problem concerns the shape of the matrix  $\by$. Relative to the number of time periods, we may have many units, few units, or a comparable number. These data configurations may make particular analyses more attractive partly by removing the need for regularization. For example, $\by$ may be a thin matrix, with $N\gg T$, or a fat matrix, with $N\ll T$, or an approximately  square matrix, with $N\approx T$:
	%
	{\small
		\[ \by=\left(
		\begin{array}{ccccccc}
		? &  \tick  &  ? \\
		\tick & ?    & \tick  \\
		?   & ?    & \tick  \\
		\tick  & ?    & \tick  \\
		?  & ?    & ?  \\
		\vdots  & \vdots  &\vdots \\
		?   & ?   &  \tick  \\
		\end{array}
		\right)\hskip0.3cm ({\bf thin})
		\hskip1cm
		\by=\left(
		\begin{array}{ccccccc}
		? & ? & \tick & \tick & \tick  & \dots & ? \\
		\tick  & \tick & \tick& \tick& ?   & \dots & \tick  \\
		?  & \tick & ? & \tick & ?  & \dots & \tick  \\
		\end{array}
		\right)\hskip0.2cm ({\bf fat}),\]
	}
	or
	{\small
		\[\by=\left(
		\begin{array}{cccccccc}
		? & ? & \tick & \tick &   \dots & ?   \\
		\tick  & \tick & \tick&  \tick   & \dots & \tick    \\
		?  & \tick & ? & \tick &    \dots & \tick   \\
		\vdots & \vdots& \vdots&  \vdots&\ddots& \vdots\\
		?  & ? &\tick & \tick &  \dots & \tick   \\
		\end{array}
		\right)\hskip0.2cm ({\bf approximately\ square}).\]
	}

	\subsection{Horizontal and Vertical  Regressions}

	Two special combinations of  missing data patterns and matrix shape deserve particular attention because they are the focus of large, mostly separate, literatures.
	
	\subsubsection{Horizontal Regression and the Unconfoundedness Literature}
	
	The unconfoundedness literature (\citet{rosenbaum1983central, rubin2006matched, imbenswooldridge, abadie2018econometric})
	focuses primarily on the single-treated-period block structure with a thin matrix ($N\gg T$), a substantial number of treated and control units, and imputes the missing potential outcomes in the last period using control units with similar lagged outcomes:
	{\small
		\[ \by=\left(
		\begin{array}{ccccccc}
		\tick & \tick& \tick  \\
		\vdots  & \vdots  &\vdots \\
		\tick  & \tick    & \tick  \\
		\tick & \tick    & ?  \\
		\vdots  & \vdots  &\vdots \\
		\tick   & \tick&  ?  \\
		\end{array}
		\right),\]
	}
	A simple version of the unconfoundedness approach  is to
	regress the last period outcome on the lagged outcomes and use the estimated regression to predict the missing potential outcomes. That is, for the   units with $(i,T)\in\calm$, the predicted outcome is
	\begin{equation}\label{horizontal}
		\hat Y_{iT}=\hat\beta_0+\sum_{s=1}^{T-1} \hat\beta_s Y_{is},
		\hskip0.3cm
		{\rm where}\
		%
		\hat\beta
		=\argmin_{\beta}\sum_{i:(i,T)\in\calo} \left( Y_{iT}-
		\beta_0-\sum_{s=1}^{T-1} \beta_s Y_{is}\right)^2.\end{equation}
	We refer to this as a {\bf horizontal} regression, where the rows of the $\by$ matrix form the units of observation.
	A more flexible, nonparametric, version of this estimator would correspond to matching where we find for each treated unit $i$ a corresponding control unit $j$ with $Y_{jt}$ approximately equal to $Y_{it}$ for all pre-treatment periods $t=1,\ldots,T-1$.
	
	\subsubsection{Vertical Regression and the Synthetic Control Literature}
	
	The synthetic control literature (\citet{abadie2010synthetic}) focuses primarily on the single-treated-unit block structure with a relatively fat $(T\gg N$) or approximately square matrix $(T\approx N$), and a substantial number of pre-treatment periods:
	{\small
		\[
		\by=\left(
		\begin{array}{ccccccc}
		\tick & \tick & \dots & \tick & \tick  & \dots & \tick\\
		\tick  & \tick & \dots& \tick& \tick& \dots & \tick  \\
		\tick  & \tick & \dots & \tick & ?  & \dots & ? \\
		\end{array}\right).\]
	}
	\cite{doudchenko} and \cite{ferman2019synthetic} show how  the Abadie-Diamond-Hainmueller synthetic control method can be interpreted as regressing the outcomes for the treated unit prior to the treatment on the outcomes for the control units in the same periods.
	That is, for the treated  unit in period $t$, for $t=T_0,\ldots,T$, the predicted outcome is
	\begin{equation}\label{vertical}\hat Y_{Nt}=\hat\gamma_0+\sum_{i=1}^{N-1} \hat\gamma_i Y_{it},\hskip0.3cm {\rm
			where}
		\ \  \hat\gamma
		=\argmin_{\gamma}\sum_{t: (N,t)\in\calo} \left( Y_{Nt}-
		\gamma_0-\sum_{i=1}^{N-1} \gamma_i Y_{it}\right)^2.\end{equation}
	We refer to this as a {\bf vertical} regression, where the columns of the $\by$ matrix form the units of observation.
	As shown in
	\citet{doudchenko}, this is generalization of the original \citet{abadie2010synthetic} synthetic control estimator, relaxing two restriction: $(i)$ that the coefficients are nonnegative and $(ii)$ that the intercept in this regression is zero.
	Note that these restrictions may well be substantively plausible and they can greatly improve precision.
	
	Although this does not appear to have been pointed out previously,
	a matching version of this estimator would correspond to finding, for each period $t$ where  unit $N$ is treated, a corresponding period $s\in\{1,\ldots,T_0\}$ such that  $Y_{is}$ is approximately equal to $Y_{Ns}$ for all control units $i=1,\ldots,N-1$. This matching version of the synthetic control estimator may serve to clarify the link between the treatment effect literature under unconfoundedness and the synthetic control literature.

	Suppose that the only missing entry is in the last period for unit $N$, period $T$. In that case
	if we estimate the horizontal regression in (\ref{horizontal}), it is still the case that imputed $\hat Y_{NT}$ is linear in the observed $Y_{1T},\ldots,Y_{N-1,T}$, just with different weights than those obtained from the vertical regression. Similarly, if we estimate the vertical regression in (\ref{vertical}), it is still the case that  $\hat Y_{NT}$ is linear in $Y_{N1},\ldots,Y_{N,T-1}$, just with different weights from the horizontal regression in (\ref{horizontal}). Note also that the restrictions that the coefficients are nonnegative and sum to one are common in the synthetic control literature, but could also be imposed in the unconfoundedness literature, although they do not appear to have been used there.
	
	Juxtaposing the unconfoundedness and synthetic control approaches as we have done here raises the question how they are related, and whether there is an approach that avoids the  choice between focusing on the cross-section and time-series correlation patterns.
	We further elaborate on the connection between the horizontal and vertical regression in \S \ref{relation} after introducing a third approach.
	
	\subsection{Fixed Effects and Factor Models}
	
	The horizontal regression focuses on a pattern in the time path of the outcome $Y_{it}$, specifically the relation between $Y_{iT}$ and the lagged $Y_{it}$ for $t=1,\ldots,T-1$, for the units for whom these values are observed, and assumes that this pattern is the same for units with missing outcomes. The vertical regression focuses on a pattern between units at times when we observe all outcomes, and assumes this pattern continues to hold for periods when some outcomes are missing.
	However, by focusing on only one of these patterns, cross-section or time series, these approaches ignore alternative patterns that may help in imputing the missing values.
	An alternative is to consider approaches that allow for the exploitation of both stable patterns over time, and stable patterns accross units. Such methods have a long history in the  panel data literature, including the literature on  two-way fixed effects, and more generally, factor and interactive fixed effect models
	(e.g., \citet{chamberlain1984panel,  angristpischke, arellano2001panel, liang1986longitudinal, bai2003inferential, bai2009panel, bai2002determining, pesaran2006estimation, moon2015linear, moon2017dynamic}).
	In the absence of covariates (although in this literature the coefficients on these covariates are typically the primary focus of the analyses),
	the common two-way fixed effect model is
	\begin{equation} \label{eq.01} Y_{it}=\gamma_i+\delta_t+\epsilon_{it}.\end{equation}
	More general  factor models
	can be written as
	\begin{equation} \label{eq.1} Y_{it}=\sum_{r=1}^R \gamma_{ir}\delta_{tr}+\varepsilon_{it}, \hskip1cm {\rm or}
		\ \  \by=\bu\bv^\top+\bepsilon,
	\end{equation}
	where $\bu$ is $N\times R$ and $\bv$ is $T\times R$.
	Most of the early literature, \citet{anderson1958introduction} and \citet{goldberger1972structural}), focused on the thin matrix case, with $N\gg T$, where  asymptotic approximations are based on letting the number of units increase with the number of time periods fixed.  In the modern part of this literature
	\citep{bai2003inferential, bai2009panel, pesaran2006estimation, moon2015linear, moon2017dynamic, bai2017principal}
	researchers allow for more complex asymptotics with both $N$ and $T$ increasing, at rates that allow for consistent estimation of the factors $\bv$ and loadings $\bu$ after imposing normalizations.
	In this literature it is typically assumed that the number of factors $R$ is fixed, although it is not necessarily known to the researcher. Methods for estimating the rank $R$ are discussed in \citet{bai2002determining} and \citet{moon2015linear}.
	
	\citet{xu2017generalized} adapts this interactive fixed effect approach to the matrix completion problem in the special case with blocked assignment, with additional applications in
	\citet{gobillon2013regional, kim2014divorce} and \citet{hsiao2012panel}. The block structure greatly simplifies the computation of the fixed rank estimators. However, this approach is not efficient, nor computationally attractive, in settings with more complex missing data patterns.

	A closely related literature has emerged in machine learning and statistics on matrix completion \citep{srebro2005generalization,candes2009exact,candes2010thepower,keshavan2010matrixFew,keshavan2010matrixNoisy,gross2011recovering,recht2011simpler, rohde2011estimation, negahban2011estimation, negahban2012restricted, koltchinskii2011nuclear,klopp2014noisy, wang2018deconfounded}. In this literature the starting point is an incompletely observed matrix $\by$, and researchers have proposed low-rank matrix models as the basis for matrix completion, similar to (\ref{eq.1}).
	The focus is not on estimating $\bu$ and $\bv$ consistently, but on imputing the missing elements
	of $\by$.
	Instead of fixing the rank $R$ of the underlying matrix, a family of these estimators rely on regularization, and in particular nuclear norm regularization.


	
	%
	
	\section{The Matrix Completion with Nuclear Norm  Minimization Estimator}\label{sec:the-mcnnm-estimator}
	
	In the absence of covariates
	we model the $N\times T$  matrix of complete outcomes data matrix $\by$ as
	\begin{equation} \by=\bl^*+\bepsilon,\hskip1cm {\rm
			where }\hskip0.5cm
		\mathbb{E}[\bepsilon|\bl^*]=\bzero\,.
	\end{equation}
	The $\varepsilon_{it}$ can be thought of as measurement error.
	\begin{assumption}\label{a1}
		$\bepsilon$ is independent of $\bl^*$, and the elements of $\bepsilon$ are $\sigma$-sub-Gaussian and independent of each other. Recall that a real-valued random variable $\err$ is $\sigma$-sub-Gaussian if for all real numbers $t$ we have $\E[\exp(t\err)]\le \exp(\sigma^2t^2/2)$.
	\end{assumption}
	The goal is to estimate the matrix $\bl^*$. We note that here the fixed effects are absorbed in $\bl^*$ since they are two rank 1 matrices and their addition does not affect our low-rank assumption on $\bl^*$.
	
	To facilitate the characterization of the estimator,
	define for any matrix $\ba$, and given a set of pairs of indices $\calo$,  the two matrices  $\bp_\calo(\ba)$
	and  $\bp_\calo^\perp(\ba)$
	with typical elements:
	\[ \bp_\calo(\ba)_{it}=
	\left\{
	\begin{array}{ll}
	A_{it}\hskip1cm & {\rm if}\ (i,t)\in\calo\,,\\
	0&{\rm if}\ (i,t)\notin\calo\,,
	\end{array}\right.\hskip0.5cm
	{\rm and}\ \
	\bp_\calo^\perp(\ba)_{it}=
	\left\{
	\begin{array}{ll}
	0\hskip1cm & {\rm if}\ (i,t)\in\calo\,,\\
	A_{it}&{\rm if}\ (i,t)\notin\calo\,.
	\end{array}\right.
	\]
	A critical role is played by various
	matrix norms, summarized in Table \ref{matrixnorms}.
	Some of these depend on the singular values,
	where, given the full
	Singular Value Decomposition (SVD)
	$ \bl_{N\times T}=\bs_{N\times N} \bSigma_{N\times T}
	\br_{T\times T}^\top,$
	the  singular values $\sigma_i(\bl)$ are the ordered diagonal elements of $\bSigma$.
	\begin{table}[ht!]
		\caption{\sc Matrix Norms for Matrix $\bl$
		}
		\vskip1cm
		\begin{center}
			\begin{tabular}{ll}
				Schatten Norm:& $\|\bl\|^{\rm S}_p\equiv  \left(\sum_{i\in[N]} \sigma_i(\bl)^p\right)^{1/p}$, $p\in(0,\infty)$\\
				Fr\"obenius  Norm:& $\|\bl\|_F\equiv\|\bl\|^{\rm S}_2= \left(\sum_{i\in[N]} \sigma_i(\bl)^2\right)^{1/2}=\left(\sum_{i\in[N]}\sum_{t\in[T]}L_{it}^2\right)^{1/2}$\\
				Rank  Norm:& $\|\bl\|_0\equiv\lim_{p\downarrow 0}\|\bl\|^{\rm S}_p= \sum_{i\in[N]} {\bf 1}_{\sigma_i(\bl)>0}$\\
				Nuclear  Norm:& $\|\bl\|_*\equiv\|\bl\|^{\rm S}_1 =\sum_{i\in[N]} \sigma_i(\bl)$\\
				Operator  Norm:& $\|\bl\|_\oper\equiv \lim_{p\rightarrow\infty} \|\bl\|^{\rm S}_p=\max_{i\in[N]}\sigma_i(\bl)=\sigma_1(\bl)$\\
				Max  Norm:& $\|\bl\|_{\max}\equiv \max_{(i,t)\in[N]\times[T]} |L_{it}|$\\
				Element-Wise $\ell_1$ Norm:~~~~~~ & $\|\bl\|_{1,e}\equiv\sum_{(i,t)\in[N]\times[T]} |L_{it}|$
			\end{tabular}
		\end{center}
		\label{matrixnorms}
	\end{table}
	Now consider the problem of estimating $\bl^*$.
	Directly minimizing the sum of squared differences,
	\begin{equation}\label{noreg}
		\min_{\bl}\fncalo\sum_{(i,t)\in\calo}\left(Y_{it}-L_{it}\right)^2
		=\min_{\bl}\fncalo\left\|\bp_\calo(\by-\bl)\right\|^2_F\,,
	\end{equation}
	does not lead to a useful estimator: if $(i,t)\in\calm$ the objective function does not depend on $L_{it}$, and for  pairs $(i,t)\in\calo$ the estimator would simply be $Y_{it}$.
	To address this we regularize the problem by adding a penalty term $\lambda\|\bl\|$ to the objective function in (\ref{noreg}), for some choice of the norm $\|\cdot\|$ and a scalar $\lambda$. However, since we don not wish to regularize the fixed effects (that are included in $\bl^*$), we estimate them explicitly by introducing variables $\Gamma\in\reals^{N\times 1}$ and $\Delta\in\reals^{T\times 1}$, and the variable $\bl$ will be used for estimating the remaining low-rank components of $\bl^*$. This is conceptually similar to not regularizing the intercept term in LASSO estimator, to reduce the bias created by the regularization term \citep{hastie2009elements}.
	
	\paragraph{The estimator:} The general form of our proposed estimator for $\bl^*$ is $\hat\bl+\hat\Gamma\id_T^\top+\id_N\hat\Delta^\top$ where
	\begin{equation}\label{objective_function1}
		(\hat\bl,\hat\Gamma,\hat\Delta)=
		\arg\min_{\bl,\Gamma,\Delta}\left\{\fncalo\|\bp_\calo(\by-\bl-\Gamma\id_T^\top-\id_N\Delta^\top)\|^2_F+\lambda\|\bl\|_*\right\}
		\,.
	\end{equation}
	Compared to the setting
	studied by \cite{candes2009exact, candes2010matrix, mazumder2010spectral},
	we include the fixed effects $\Gamma$ and $\Delta$.
	Although formally the fixed effects can be subsumed in the matrix $\bl$ ($\Gamma\id_T^\top$ and $\id_N
	\Delta^\top$ are both rank one matrices), in practice, including these fixed effects separately and not regularizing them greatly improves the quality of the imputations. This is partly because compared to the settings studied in the  matrix completion literature the fraction of observed values is relatively high, and so these fixed effects can be estimated accurately.
	The penalty factor $\lambda$ will be chosen through cross-validation that will be described at the end of this section.
	We will call this the Matrix-Completion with Nuclear Norm Minimization (MC-NNM) estimator.
	
	Other commonly used Schatten norms would not work as well for this specific problem.
	For example, the Fr\"obenius norm on the penalty term would not have been suitable for estimating $\bl^*$ in the case with missing entries because the solution for $L_{it}$ for $(i,t)\in\calm$ is always zero (which follows directly from the representation of $\|\bl\|_F=\sum_{(i,t)\in[N]\times[T]}L_{it}^2$). The rank norm
	is not computationally feasible for large $N$ and $T$ if the cardinality and complexity of the set $\calm$ are substantial. Formally, the optimization problem withs the rank norm is NP-hard.
	In contrast, a major advantage of using the nuclear norm is that the resulting estimator can be computed using fast convex optimization programs, e.g. the {\sc soft-impute} algorithm by \citet{mazumder2010spectral} that will be described next.
	
	\paragraph{Calculating the Estimator:} For simplicity, let us first assume that there are no fixed effects (so that we do not need to estimate $\Gamma$ and $\Delta$). The algorithm for calculating our estimator  goes as follows.
	Given the SVD for $\ba$,
	$\ba=\bs\bSigma\br^\top$,
	with singular values $\sigma_1(\ba),\ldots,\sigma_{\min(N,T)}(\ba)$,
	define the matrix shrinkage operator
	\begin{equation}\label{shrinkage}
		\shrink_\lambda(\ba)=\bs \tilde\bSigma\br^\top\,,
	\end{equation}
	where $\tilde\bSigma$ is equal to $\bSigma$ with the $i$-th  singular value $\sigma_i(\ba)$ replaced by $\max(\sigma_i(\ba)-\lambda,0)$. Now start with the initial choice $\bl_1(\lambda,\calo)=\bp_\calo(\by)$.
	Then for $k=1,2,\ldots,$ define,
	\begin{equation}\label{eq:shrink_operator}
		\bl_{k+1}(\lambda,\calo)=\shrink_{\frac{\lambda|\calo|}{2}}\Bigl\{\bp_\calo(\by)+\bp_\calo^\perp\Big(\bl_k(\lambda,\calo)\Big)\Bigr\}\,,
	\end{equation}
	until the sequence $\left\{\bl_k(\lambda,\calo)\right\}_{k\ge 1}$ converges. The limiting matrix $\hat\bl(\lambda,\calo)=\lim_{k\rightarrow\infty}\bl_k(\lambda,\calo) $ is our estimator given the regularization parameter $\lambda$. For the case that we are estimating fixed effects, after each iteration of obtaining $\bl_{k+1}$, we can estimate $\Gamma_{k+1}$ and $\Delta_{k+1}$ by using the first order conditions since they only appear in the squared error term. We would also replace the $\bp_\calo(\by)$ term in \eqref{eq:shrink_operator} by $\bp_\calo(\by-\Gamma_k\id_T^\top-\id_N\Delta_k^\top)$.
	
	\paragraph{Cross-validation:} The optimal value of $\lambda$ is selected through cross-validation. We choose $K$ (e.g., $K=5)$ random subsets  $\calo_k\subset\calo$ with cardinality $\lfloor|\calo|^2/NT\rfloor$ to ensure that the fraction of observed data in the cross-validation data sets, $|\calo_k/|\calo|$, is equal to that in the original sample, $|\calo|/(NT)$. We then select a sequence of candidate regularization parameters
	$\lambda_1>\cdots>\lambda_L=0$,
	with a large enough $\lambda_1$, and for each subset $\calo_k$ calculate
	$
	\hat\bl(\lambda_1,\calo_k),\ldots,\hat\bl(\lambda_L,\calo_k)
	$
	and evaluate the average squared error  on $\calo\setminus\calo_k$. The value of $\lambda$ that minimizes the average squared error (among the $K$ produced estimators corresponding to that $\lambda$) is the one chosen. It is worth noting that one can expedite the computation by using $\hat\bl(\lambda_i,\calo_k)$ as a warm-start initialization for calculating $\hat\bl(\lambda_{i+1},\calo_k)$ for each $i$ and $k$.
	
	\paragraph{Confidence intervals:} Studying asymptotic distribution of $\bl^*-\hat\bl$ in order to construct confidence intervals is beyond the scope of this paper and is an interesting future research question. However, one can use re-sampling methods to view statistical fluctuations of the imputed matrix. For example, one can again choose $K$ random subsets  $\calo_k\subset\calo$ and construct a cross-validated estimator $\hat\bl^{(k)}$ for each set $\calo_k$. Then, for each entry $(i,t)$ use statistical fluctuations of $\{\hat L_{it}^{(k)}\}_{k\in[K]}$ to construct a confidence interval for $L^*_{it}$, related to the use of permutation metods in the synthetic control literature (\citet{abadie2010synthetic}).


	\section{The Relationship with Horizontal and Vertical Regressions}\label{relation}
	
	In the second contribution of this paper we discuss the relation between the matrix completion estimator and  the horizontal (unconfoundedness), vertical (synthetic control) and difference-in-differences approaches. To faciliate the discussion, we focus on the case with the set of missing pairs $\calm$ containing a single pair, unit $N$ in period $T$,  $\calm=\{(N,T)\}$. In that case the various previously proposed versions of the vertical and horizontal regressions are both directly applicable, although estimating the coefficients may require regularization depending on the relative magnitude of $N$ and $T$.
	
	The observed data are $\by$, an $N\times T$ matrix with the $(N,T)$ entry missing. We can partition this matrix as
	\[ \by=\left(
	\begin{array}{cc}
	\ty & \bye\\
	\bytt & ?
	\end{array}\right),\]
	where $\ty$ is a $(N-1)\times (T-1)$ matrix, and
	$\bye$ and $\byt$ are  $(N-1)$ and $(T-1)$ component vectors, respectively.
	
	In this case the matrix completion, horizontal regression, vertical regression,  synthetic control regression, the elastic net version, and difference-in-differences estimators  are very closely related. They can all be charactized as focusing on the exact same objective function, but differing in the regularization and additional restrictions imposed on the parameters of the objective function.
	
	To make this precise,
	define for a given positive integer $R$, an $N\times R$ matrix $\ba$, an $T\times R$ matrix $\bb$, an $N$-vector $\gamma$ and a $T$-vector $\delta$ the objective  function
	\begin{equation}\label{obj} Q(\by;R,\ba,\bb,\gamma,\delta)=
		\fncalo\left\|P_\calo \left(\by-\ba\bb^\top-\gamma\id_T^\top-\id_N\delta^\top\right)\right\|^2_F
		\end{equation}
	For any pair of positive integers $K$ and $L$, let $\mmm^{K,L}$ be the set of all $K\times L$ real-valued matrices. When $R=0$, we take the product $\ba\bb^\top$ to be the $N\times T$ matrix with all elements equal to zero.
	First note that simply minimizing $Q(\by;R,\ba,\bb,\gamma,\delta)$  over  the rank $R$, the matrices $\ba$, $\bb$ and the  vectors $\gamma$ and $\delta$,
	\[
	\min_{R\in\{0,1,\ldots,\min(N,T)\}}~~\min_{\ba\in\mmm^{N,R},\bb\in\mmm^{T,R},\gamma\in\mmm^{N,1},\delta\in\mmm^{T,1}}
	Q(\by;R,\ba,\bb,\gamma,\delta),\]
	has multiple  solutions for the imputations $\hat\by_{NT}$  where $\hat\by=\ba\bb^\top+\gamma\id_T^\top+\id_N\delta^\top$. By choosing the rank $R$ to the minimum of $N$ and $T$, we can find for any pair $\gamma$ and $\delta$ a solution for $\ba$ and $\bb$ such that
	$P_\calo \left(\by-\ba\bb^\top-\gamma\id_T^\top-\id_N\delta^\top\right)$ has all elements equal to zero, with different values for $\hat\by_{NT}$.
	
	The implication is that we need to add some structure to the optimization problem. The next result shows that horizontal regression, vertical regression, the Abadie-Diamond-Hainmueller synthetic control estimator, the difference-in-differences estimator, and the nuclear norm minimization matrix  completion can all be expressed as minimizing $Q(\by;R,\ba,\bb,\gamma,\delta)$ under different restrictions on,  or with different approaches to regularization of the unknown parameters $(R,\ba,\bb,\gamma,\delta)$.
	The following theorem lays out these differences in hard restrictions and regularization approaches. Here the minimization for $R$ is over the set $\{0,1,2,\ldots,\min(T,N)\}$, and the minimization for $\ba$ and $\bb$ is over the sets $\mmm^{N,R}$ and $\mmm^{T,R}$ respectively.

	\begin{theorem}\label{thm:representation2}
		In the case with only the $(N,T)$ entry missing, we have,\\
		$(i)$ (nuclear norm matrix completion)
		\begin{multline*} (R^\mc,\ba^\mc_\lambda,\bb^\mc_\lambda,\gamma^\mc_\lambda,\delta^\mc_\lambda)=\\
		\text{\emph{argmin}}_{R,\ba,\bb,\gamma,\delta}\left\{
		Q(\by;R,\ba,\bb,\gamma,\delta)
		+\frac{\lambda}{2}\|\ba\|^2_F+\frac{\lambda}{2}\|\bb\|^2_F\right\},
		\end{multline*}
		$(ii)$ (horizontal regression, defined if $N>T$)
		\[(R^\hr,\ba^\hz,\bb^\hz,\gamma^\hz,\delta^\hz)=\text{\emph{argmin}}_{R,\ba,\gamma,\delta}
		Q(\by;R,\ba,\bb,\gamma,\delta)
		,\]
		subject to
		\[
		R=T-1,\hskip0.5cm
		\ba=
		\left(
		\begin{array}{c}
		\ty \\
		\bytt
		\end{array}\right),\hskip0.5cm \gamma=0, \hskip0.5cm \delta_1=\delta_2=\ldots=\delta_{T-1}=0,
		\]
		$(iii)$ (vertical regression, defined if $T>N$),
		\[(R^\vt,\ba^\vt,\bb^\vt,\gamma^\vt,\delta^\vt)=\text{\emph{argmin}}_{R,\ba,\bb,\gamma,\delta}
		Q(\by;R,\ba,\bb,\gamma,\delta)
		,\]
		subject  to
		\[R=N-1, \hskip0.5cm \bb=
		\left(
		\begin{array}{cc}
		\ty^\top \\
		\bye^\top
		\end{array}\right), \hskip0.5cm \gamma_1=\gamma_2=\ldots=\gamma_{N-1}=0,\hskip0.5cm
		\delta=0,
		\]
		$(iv)$ (synthetic control),
		\begin{multline*}
			(R^\scc,\ba^\scc,\bb^\scc,\gamma^\scc,\delta^\scc)=
		\text{\emph{argmin}}_{R,\ba,\bb,\gamma,\delta}
		Q(\by;R,\ba,\bb,\gamma,\delta)\,,
		\end{multline*}
		subject  to
		\[R=N-1, \hskip0.5cm \bb=
		\left(
		\begin{array}{cc}
		\ty^\top \\
		\bye^\top
		\end{array}\right),\hskip0.5cm
		\delta=0, \hskip0.5cm \gamma=0, \hskip0.5cm
		\forall\ i, A_{iT}\geq 0,\ \sum_{i=1}^{N-1} A_{iT}=1,
		\]
		$(v)$ (vertical regression, elastic net),
		\begin{multline*}(R^\divv,\ba^\divv,\bb^\divv,\gamma^\divv,\delta^\divv)=\\
		\text{\emph{argmin}}_{R,\ba,\bb,\gamma,\delta}\Bigg\{
		Q(\by;R,\ba,\bb,\gamma,\delta)
		+\lambda \left[\frac{1-\alpha}{2}\left\|\left(
		\begin{array}{c}
		\bat \\ \batth
		\end{array}\right)\right\|^2_F
		+
		\alpha\left\|\left(
		\begin{array}{c}
		\bat \\ \batth
		\end{array}\right)\right\|_1
		\right]\Bigg\}
		,\end{multline*}
		subject  to
		\[R=N-1, \hskip0.5cm \bb=
		\left(
		\begin{array}{cc}
		\ty^\top \\
		\bye^\top
		\end{array}\right), \hskip0.5cm \gamma_1=\gamma_2=\ldots=\gamma_{N-1}=0,\hskip0.5cm
		\delta=0,
		\]
		where $\ba$ is partitioned as
		\[
		\ba=\left(
		\begin{array}{cc}
		\tba & \bae\\
		\batt & \batth
		\end{array}\right),\]
		
		$(vi)$ (difference-in-differences regression),
		\[(R^\ddd,\ba^\ddd,\bb^\ddd,\gamma^\ddd,\delta^\ddd)=\text{\emph{argmin}}_{R,\ba,\bb,\gamma,\delta}
		Q(\by;R,\ba,\bb,\gamma,\delta)
		,\]
		subject  to
		\[R=0.
		\]
	\end{theorem}
	The proof for this result is in \S \ref{subsec:representation2Proof}.
	
	\begin{comments}{\rm There is no unique solution to minimizing $Q(\by;\ba,\bb)$ if we also minimize over the rank $R$. The nuclear norm estimator uses regularization to get around this by regularizing $\ba$ and $\bb$. The other estimators impose restrictions instead of (or in combination with) regularizing the estimators, while fixing $R$ as a function of $N$ and $T$. The restrictions for the horizontal regression on the one hand, and for the vertical regression, synthetic control and elastic net regression on the other hand,  are quite different, and not directly comparable. However in other settings researchers have found that it is often better to regularize estimators than to impose hard restrictions. We find the same in our simulations below.}\end{comments}
	
	\begin{comments}{\rm For nuclear norm matrix completion representation
			a key insight is that  (Lemma 6, \citet{mazumder2010spectral})
			\[ \left\|\bl\right\|_*=\min_{\ba,\bb:\bl=\ba\bb^\top} \frac{1}{2}\left(
			\left\|\ba\right\|^2_F+\left\|\ba\right\|^2_F
			\right).\]
			In addition, if $\hat{\bl}$ is the solution to Equation $\eqref{objective_function1}$ that has rank $\hat{R}$, then
			one solution for $\ba$ and $\bb$ is given by
			\begin{equation}\label{eq:ABt_for_NN}
				\ba = \bs \bSigma^{1/2}~,~~~\bb = \br \bSigma^{1/2}
			\end{equation}
			where $\hat{\bl}=\bs_{N\times \hat{R}} \bSigma_{\hat{R}\times \hat{R}}
			\br_{T\times \hat{R}}^\top$ is singular value decomposition of $\hat\bl$. The proof of this fact is provided in (\citet{mazumder2010spectral,hastie2015matrix}).
			$\square$} \end{comments}

	\begin{comments}{\rm
			For the horizontal regression the solution for $\bb$ is
			\[ \bb^\hz=\left(
			\begin{array}{cccc}
			1 & 0 & \ldots & 0 \\
			0 & 1 & \ldots & 0 \\
			\vdots & \vdots && \vdots   \\
			0 & 0 & \ldots & 1  \\
			\hat\beta_1 & \hat\beta_2 & \ldots & \hat\beta_{T-1}
			\end{array}\right),\]
			where $\hat\beta$  is
			\[ (\hat\beta,\hat\delta_T)=\arg\min_{\beta,\delta_T}\sum_{i=1}^{N-1}\left(
			Y_{iT}-\delta_T-\sum_{t=1}^{T-1}\beta_t Y_{it}\right)^2.\]
			Similarly for the vertical regression the solution for $\ba$ is
			\[ \ba^\vt=\left(
			\begin{array}{cccc}
			1 & 0 & \ldots & 0 \\
			0 & 1 & \ldots & 0 \\
			\vdots & \vdots && \vdots   \\
			0 & 0 & \ldots & 1  \\
			\hat\alpha_1 & \hat\alpha_2 & \ldots & \hat\alpha_{N-1}
			\end{array}\right),\]
			where \[ (\hat\alpha,\hat\gamma_N)=\arg\min_{\alpha,\gamma_N}\sum_{t=1}^{T-1}\left(
			Y_{Nt}-\gamma_N-\sum_{i=1}^{N-1}\alpha_i Y_{it}\right)^2.\]
			The regularization in the elastic net version only affects the last row of this matrix, and replaces it with a regularized version of the regression coefficients. The synthetic control estimator further restricts the values of the $\gamma_N$ and $\alpha_i$.
			$\square$} \end{comments}

	\begin{comments}{\rm
			The horizontal and vertical regressions are fundamentally different approaches, and they cannot easily be nested. Without some form of regularization they cannot be applied in the same setting, because the  non-regularized versions require $N>T$ or $N<T$ respectively. As a result there is also no direct way to test the two methods against each other. Given a particular choice for regularization, however, one can use cross-validation methods to compare the two approaches.
			$\square$} \end{comments}

	%
	
	\section{Theoretical Bounds for the Estimation Error}\label{sec:estimation-error}
	
	In this section we focus on the case that there are no covariates or fixed effects, and provide theoretical results for the estimation error.
	Let $\lmax$ be a positive constant such that $\|\bl^*\|_{\mmax}\le \lmax$ (recall that $\|\bl^*\|_{\mmax}=\max_{i,t}|\bl_{it}^*|$). We also assume that $\bl^*$ is a deterministic matrix.
	Then consider the following estimator for $\bl^*$.
	\begin{equation}\label{objective_function}
		\hat\bl=\argmin_{\substack{\bl: \|\bl\|_{\mmax}\le \lmax}}\left \{\frac{1}{\ncalo}{\|\bp_\calo(\by-\bl)\|^2_F}
		+\lambda\|\bl\|_*\right\}\,.
	\end{equation}
	
	%

	\subsection{Additional Notation}\label{notation}
	
	First, we start by introducting some new notation.
	Recall that for each positive integer $n$ notation $[n]$ refers to the set of integers $\{1,2,\ldots,n\}$. For any two real numbers $a$ and $b$, we denote their maximum by $a\vee b$. In addition, for any pair of integers $i,n$ with $i\in[n]$ define  $e_i(n)$ to be the $n$ dimensional column vector with all of its entries equal to $0$ except the $i^{th}$ entry that is equal to $1$. In other words, $\{e_1(n), e_2(n),\ldots,e_n(n)\}$ forms the standard basis for $\reals^n$.
	For any two matrices $\ba,\bb$ of the same dimensions define the inner product
	$\langle \ba,\bb \rangle\equiv\tr(\ba^\top \bb)$. Note that with this definition, $\langle \ba,\ba\rangle=\|\ba\|_F^2$.

	Next, we describe a random observation process that defines the set $\calo$. Consider $N$ independent random variables $\{t_i\}_{i\in[N]}$ on $[T]$ with distributions $\{\pi^{(i)}\}_{i\in[N]}$.
	Specifically, for each $(i,t)\in[N]\times[T]$, define $\pi_{t}^{(i)}\equiv\prob[t_i=t]$.
	We also use the short notation $\E_\pi$ when taking expectation with respect to all distributions $\{\pi^{(i)}\}_{i\in[N]}$. Now, $\calo$ can be written as
	$
	\calo = \bigcup_{i=1}^N \Big\{(i,1),(i,2),\ldots,(i,t_i)\Big\}$. The equivalent of the unconfoundedness assumption in the program evaluation literature is that the adoption dates are independent of each other and of the idiosyncratic part of the outcomes, conditional on the systematic part. Formally, we make the following assumption:
	\begin{assumption}\label{a2} Conditional on $\bl^*$, the adoption dates $t_i$ are independent of each other and of  $\bepsilon$.
	\end{assumption}
	\begin{remark}
			This assumption is similar to the unconfoundedness assumption. In the setting where researchers use that assumption, with a single treated period, the only stochastic component of $\bw$ is the last column. In that case the assumption is that conditional on the first $T-1$ rows of $\by$, the last column of the assignment $\bw$ is independent of the last column of $\by$. As we show in \S
			\ref{relation}, in the unconfoundedness approach the first $T-1$ columns of the matrix $\bl$ are taken to be identical to the first $T-1$ columns of the matrix $\by$ (and the last column of $\bl$ is a linear combination of the first $T-1$ columns), so the conditioning on the first $T-1$ columns of $\by$ is identical to conditioning on $\bl$.
	\end{remark}
	
	Also, for each $(i,t)\in\calo$, we use the notation $\ba_{it}$ to refer to $e_{i}(N)e_{t}(T)^\top$ which is a $N$ by $T$ matrix with all entries equal to zero except the $(i,t)$ entry that is equal to $1$. The data generating model can now be written as
	\[
	Y_{it}=\langle \ba_{it}, \bl^* \rangle + \err_{it}\,,~~~ \forall ~(i,t)\in\calo\,,
	\]
	where noise variables $\err_{it}$ satisfy Assumptions \ref{a1}-\ref{a2}.
	
	Note that the number of control units ($\nc$) is equal to the number of rows that have all entries observed (i.e., $\nc=\sum_{i=1}^N\ind_{\{t_i=T\}}$). Therefore, the expected number of control units can be written as
	$\E_\pi[\nc]=\sum_{i=1}^N\pi_T^{(i)}$. Defining
	\[
	\pc \equiv \min_{1\le i\le N}\pi_T^{(i)}\,,
	\]
	we expect to have (on average) at least $N \pc$ control units. The parameter $\pc$ will play an important role in our main theoretical results.  To provide some intuition, assume $\bl^*$ is a matrix that is zero everywhere except in its $i^{th}$ row. Such $\bl^*$ is clearly low-rank. But recovering the entry $L^*_{iT}$ is impossible when $i_t<T$ which means $\pi_T^{(i)}$ cannot be too small. Since $i$ is arbitrary, in general, $\pc$ cannot be too small.
	\begin{remark}\label{rem:1}
		It is worth noting that the sources of randomness in our observation process $\calo$ are the random variables $\{t_i\}_{i=1}^N$ that are assumed to be independent of each other. But we allow that distributions of these random variables to be functions of $\bl^*$. We also assume that the noise variables $\{\err_{it}\}_{it\in[N]\times[T]}$ are independent of each other and are independent of $\{t_i\}_{i=1}^N$. In \S \ref{sec:future_theory} we discuss how our results could generalize to the cases with correlations among these noise variables.
	\end{remark}
	\begin{remark}\label{rem:2}
		The estimator \eqref{objective_function} penalizes the error terms $(Y_{it}-L_{it})^2$, for $(i,t)\in\calo$, equally. But the ex ante probability of missing entries in each row, the propensity score,  increases as $t$ increases. In \S \ref{subsec:weightedLossFunction}, we discuss how the estimator can be modified by considering a weighted loss function based on propensity scores for the missing entries.
	\end{remark}
	
	\subsection{Main Result}\label{subsec:consistency_result}
	
	The main result of this section is the next theorem (proved in \S \ref{subsec:consistencyProof}) that provides an upper bound for
	$\|\bl^*-\hat\bl\|_F/\sqrt{NT}$, the root-mean-squared-error (RMSE) of the estimator $\hat\bl$.
	\begin{theorem}\label{thm:consistency}
		Suppose Assumptions \ref{a1} and \ref{a2} hold,  rank of $\bl^*$ is $R$, $T\ge C_0\log(N+T)$ for a constant $C_0$, and the penalty parameter $\lambda$
		is a constant multiple of
		\[
		\frac{\sigma\left[\sqrt{N\log(N+T)}\vee\sqrt{T\log^{3}(N+T)}\right]}{|\calo|}\,.
		\]
		Then there is a constant $C$ such that with probability greater than $1-2(N+T)^{-2}$,
		{\small
			\begin{align}
			\label{eq:consistency}
				\frac{\|\bl^*-\hat\bl\|_F}{\sqrt{NT}}	\leq
			 C\sqrt{\frac{\lmax^2\log(N+T)}{N\,\pc} \vee \left[\left(\frac{\sigma^2R\log(N+T)}{T\,\pc^2}\vee\frac{\sigma^2R\log^{3}(N+T)}{N\,\pc^2}\right)+ {\update\frac{\lmax^2}{\sqrt{N}\pc}}\right]}\,.
			\end{align}
		}
	\end{theorem}
\paragraph{Interpretation of Theorem \ref{thm:consistency}:}
In order to see when the RMSE of $\hat\bl$ converges to zero as $N$ and $T$ grow, we note that the right hand side of \eqref{eq:consistency} converges to $0$ when $\bl^*$ is low-rank ($R$ is constant), {\update and
$\pc\gg (\sqrt{1/T}\vee \sqrt{1/N})\log^{3/2}(N+T)$. For example, when $T$ is the same order as $N$,
a sufficient condition for the latter is that the lower bound for the average number of control units ($N\pc$) grows faster than a constant multiple of $\sqrt{N}\log^{3/2}(N)$.} In \S \ref{sec:future_theory} we will discuss how the estimator $\hat\bl$ should be modified to obtain a sharper result that would hold for a smaller number of control units.

	\paragraph{Comparison with existing theory on matrix-completion:} Our estimator and its theoretical analysis are motivated by and generalize existing research on matrix-completion \citep{srebro2005generalization, mazumder2010spectral, candes2009exact,candes2010thepower,keshavan2010matrixFew,keshavan2010matrixNoisy,gross2011recovering,recht2011simpler, rohde2011estimation, negahban2011estimation, negahban2012restricted, koltchinskii2011nuclear,klopp2014noisy}. The main difference is in our observation model $\calo$. Existing papers assume that entries $(i,t)\in\calo$ are independent random variables whereas we allow for a time series  dependency structure. In particular this includes the staggered adoption setting where if $(i,t)\in\calo$ then $(i,t')\in\calo$ for all $t'< t$. The impact of this additional correlation is that the estimation error deteriorates significantly, compared to the ones in prior literature.  For example, as discussed above, in the case of {\update $N=T$}, in order to have a consistent estimation we need more data. Specifically, a factor {\update $\sqrt{N}$} (up to logarithmic factors) more entries  per column should be observed, than in the matrix completion literature.
	\begin{remark}\label{rem:random-O} We note that in statement of Theorem \ref{thm:consistency}, the
		lower bound on $\lambda$ depends on $\calo$ which is a random variable. The left hand side of the inequality \eqref{eq:consistency} is also random, depending on $\calo$ and the noise, but the right hand side of \eqref{eq:consistency} is deterministic. In order to understand the role of randomness, we describe the main three steps of the proof. First, in Lemma \ref{lem:error_bound}, we prove a deterministic upper bound for $\sum_{(i,t)\in\calo}\langle\ba_{it},\bl^*-\hat\bl\rangle^2/|\calo|$
		that holds for every realization of the random variable $\calo$, when $\lambda$ grows by operator norm of a certain error matrix, $\sum_{(i,t)\in\calo}\err_{it}\ba_{it}$. Next, in Lemma \ref{lem:bound_err_matrix}, we use randomness of $\calo$ and noise to prove a probabilistic bound on the operator norm of this error matrix. The final step, Lemma \ref{lem:lower_bound_err}, also uses randomness of $\calo$ and noise to show that $\sum_{(i,t)\in\calo}\langle\ba_{it},\bl^*-\hat\bl\rangle^2/|\calo|$ concentrates and (with high probability) is larger than a constant fraction of its expectation up to an additive constant.
	\end{remark}

	\section{Two Illustrations}\label{sec:simulations}

	The objective of this section is to compare the accuracy of imputation for the matrix completion method with previously used methods. In particular, in a real data matrix $\by$ where no unit is treated (no entries in the matrix are missing), we choose a subset of units as hypothetical treated units and aim to predict their values (for time periods following a randomly selected initial time). Then, we report the average root-mean-squared-error (RMSE) of each algorithm on values for the pseudo-treated (time, period) pairs.
	In these cases there is not necessarily a single right algorithm. Rather, we wish to assess which of the algorithms generally performs well, and which ones are robust to a variety of settings, including different adoption regimes and different configurations of the data.
	
	We compare the following five estimators:
	\begin{itemize}
		\item \textbf{DID}: Difference-in-differences based on regressing the observed outcomes on unit and time fixed effects and a dummy for the treatment.
		\item \textbf{VT-EN}: The vertical regression with elastic net regularization, relaxing the restrictions from the synthetic control estimator.
		\item \textbf{HR-EN}: The horizontal regression with elastic net regularization, similar to unconfoundedness type regressions.
		\item \textbf{SC-ADH}: The original synthetic control approach by  \cite{abadie2010synthetic}, based on the vertical regression with Abadie-Diamond-Hainmueller restrictions. Although this estimator is not necessarily well-defined if $N\gg T$, the restrictions ensured that it was well-defined in all the settings we used.
		\item \textbf{MC-NNM}: Our proposed matrix completion approached via nuclear norm minimization, explained in \S \ref{sec:the-mcnnm-estimator} above.
		%
	\end{itemize}
	
	The comparison between \textbf{MC-NNM} and the two versions of the elastic net estimator, \textbf{HR-EN} and \textbf{VT-EN}, is particularly salient. In much of the literature researchers choose ex ante between vertical and horizontal type regressions. The   \textbf{MC-NNM} method allows one to sidestep that choice in a data-driven manner.

	\subsection{The Abadie-Diamond-Hainmueller California Smoking Data}\label{sec:smoking}
	
	We use the control units from the California smoking data studied in \cite{abadie2010synthetic} with $N=38, T=31$. Note that in the original data set there are $39$ units but one of them (state of California) is treated which will be removed in this section since the untreated values for that unit are not available.
	We then artificially designate some units and time periods to be treated, and compare predicted values for those unit/time-periods to the actual values.
	
	We consider two settings for the treatment adoption:
	\begin{itemize}
		\item Case 1: Simultaneous adoption where randomly selected $N_t$ units adopt the treatment in period $T_0+1$, and the remaining units never adopt the treatment.
		\item Case 2: Staggered adoption where randomly $N_t$ units adopt the treatment in some period after period $T$, with the actual adoption date varying randomly among these units.
	\end{itemize}
	
	In each case, the average RMSE, for different ratios $T_0/T$, is reported in Figure \ref{fig:california}.
	For clarity of the figures, for each $T_0/T$, while all 95\% sampling intervals of various methods are calculated using the same ratio $T_0/T$, in the figure they are slightly jittered to the left or right.
	In the simultaneous adoption case, DID generally does poorly, suggesting that the data are rich enough to support more complex models.
	For small values of $T_0/T$, SC-ADH and HR-EN perform poorly while VT-EN is superior. As $T_0/T$ grows closer to one, VT-EN, HR-EN, SC-ADH and MC-NNM methods all do well. The staggered adoption results are similar with some notable differences; VT-EN performs poorly (similar to DID) and MC-NNM is the superior approach. The performance improvement of MC-NNM can be attributed to its use of additional observations (pre-treatment values of treatment units).
	\begin{figure}[H]
		\begin{center}
			\begin{subfigure}{0.48\textwidth}
				\includegraphics[width=\textwidth]{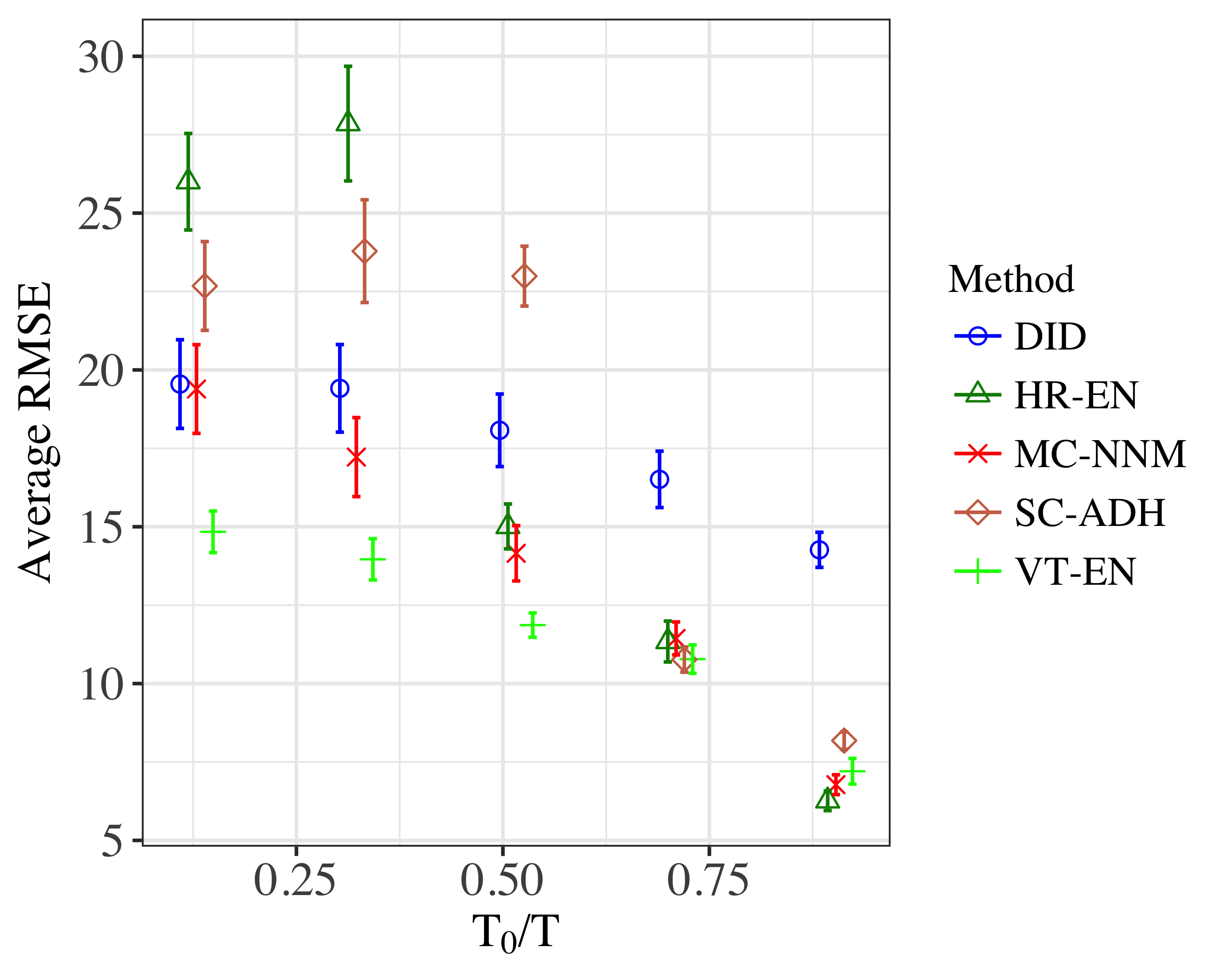}
				\caption{Simultaneous adoption, $N_t=8$}
			\end{subfigure}
			\hfill
			\begin{subfigure}{0.48\textwidth}
				\includegraphics[width=\textwidth]{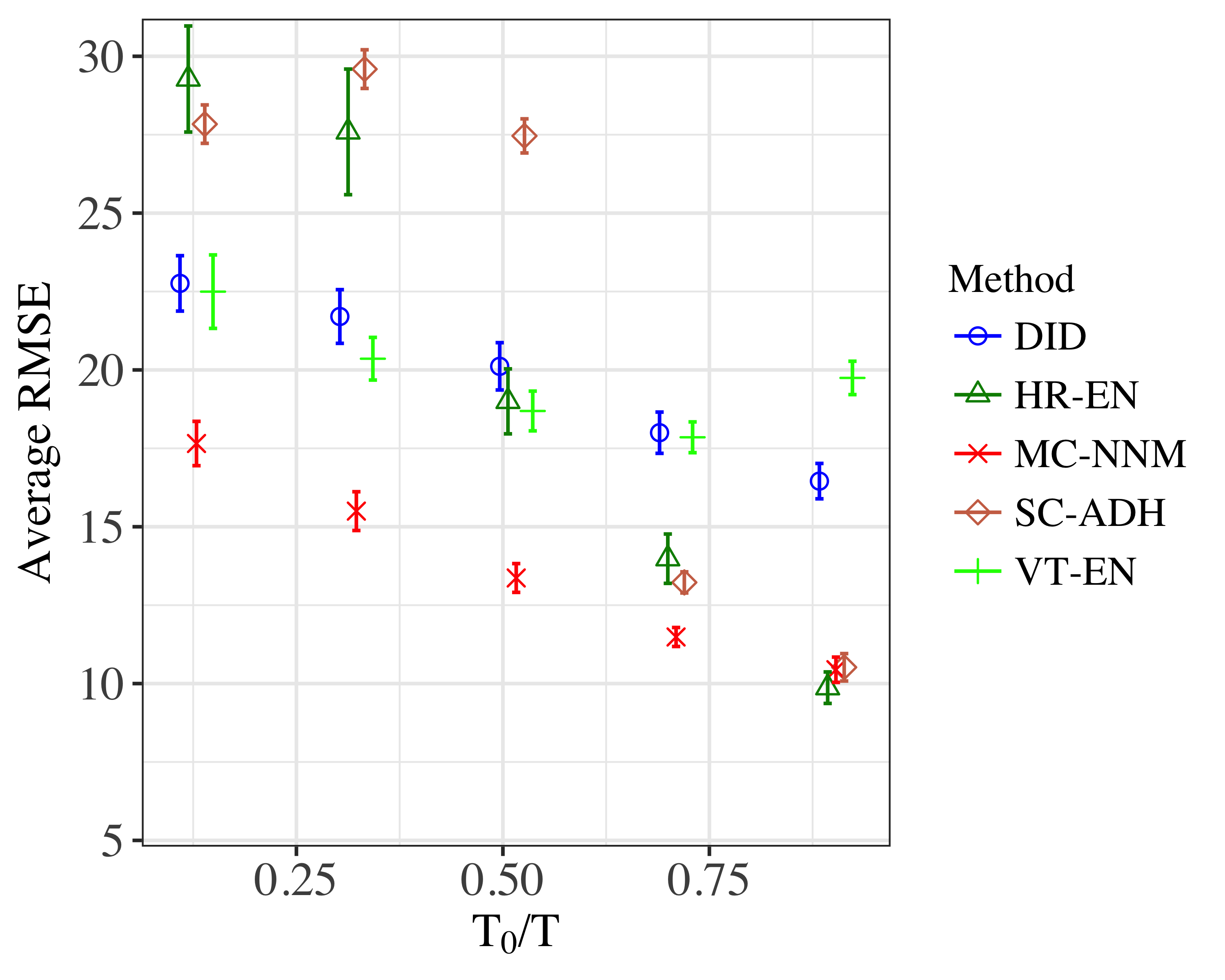}
				\caption{Staggered adoption, $N_t=35$}
			\end{subfigure}
		\end{center}
		\caption{California Smoking Data}\label{fig:california}
	\end{figure}

	\subsection{Stock Market Data}\label{sec:stocks}
	
	In the next illustration we use
	a financial data set -- daily returns for $2453$ stocks over 10 years ($3082$ days). Since we only have access to a single instance of the data, in order to observe statistical fluctuations of the RMSE, for each $N$ and $T$ we create $50$ sub-samples by looking at the first $T$ daily returns of $N$ randomly sampled stocks for a range of pairs of $(N,T)$, always with $N\times T=4900$, ranging from very thin to very fat, $(N,T)=(490,10)$, $\ldots$,
	$(N,T)=(70,70)$, $\ldots$,
	$(N,T)=(10,490)$, with in each case the second half the entries missing for a randomly selected half the units (so 25\% of the entries missing overall), in a block design.
	Here we focus on the comparison between the \textbf{HR-EN}, \textbf{VT-EN}, and \textbf{MC-NNM} estimators as the shape of the matrix changes.
	We report the average RMSE. Figure \ref{fig:stocks-data} shows the results. 
	\begin{figure}[H]
		\begin{center}
			\includegraphics[width=.8\textwidth]{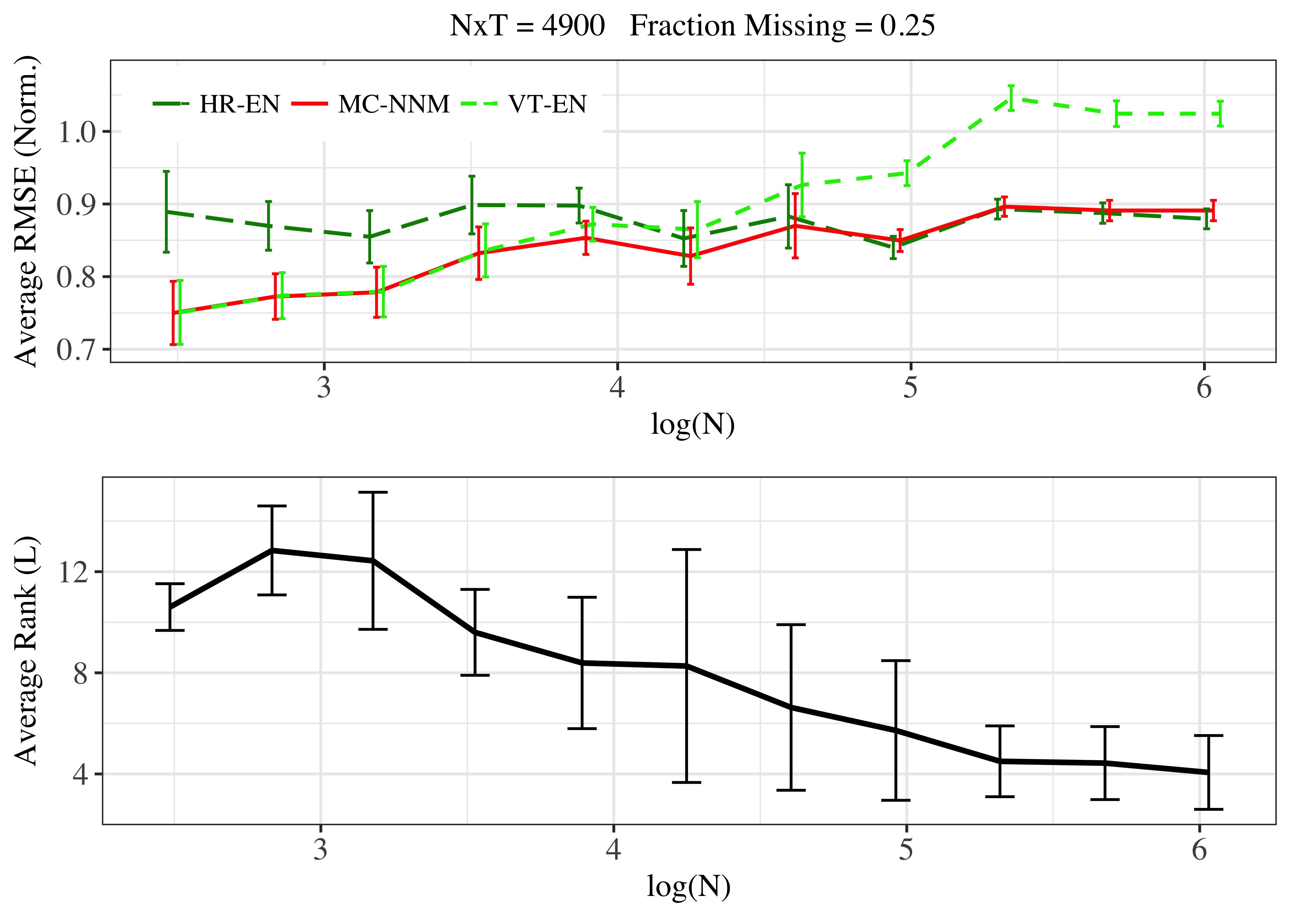}
			%
			\vspace{5mm}
			\caption{Stock Market Data}\label{fig:stocks-data}
		\end{center}
	\end{figure}
	In the $T\ll N$ case the  \textbf{VT-EN} estimator does poorly, not surprisingly  because it attempts to do the vertical regression with too few time periods to estimate that well. When $N\ll T$, the  \textbf{HR-EN} estimator does poorly for the same reason: it is trying to do the horizontal regression with too few observations relative to the number of regressors. The most interesting finding is that the proposed   \textbf{MC-NNM} method adapts well to both regimes and does as well as the best estimator in both settings, and better than both in the approximately square setting.
	
	The bottom graph in Figure \ref{fig:stocks-data} shows that MC-NNM approximates the data with a matrix of rank 4 to 12, where smaller ranks are used as $N$ grows relative to $T$. This validates the fact that there is a stronger correlation between daily return of different stocks than between returns for different time periods of the same stock.

	
	%

	\section{Generalizations}\label{sec:future_theory}
	
	Here we provide a brief discussion on how our estimator and its analysis should be adapted to more general settings.

	\subsection{The Model with Covariates}
	
	In \S \ref{section_setup} we described the basic model, and discussed the specification and estimation for the case without covariates. In this section we extend that to the case with unit-specific, time-specific, and unit-time specific covariates. For unit $i$ we observe a vector of unit-specific covariates denoted by $X_i$, and $\bx$ denoting the  $N\times P$ matrix of covariates with $i$th row equal to $X_i^\top$. Similarly, $Z_t$ denotes the time-specific covariates for period $t$, with $\bz$ denoting the $T\times Q$  matrix with $t^{\text{th}}$ row equal to $Z_t^\top$. In addition we allow for a unit-time specific $J$ by $1$ vector of covariates $V_{it}$.
	
	The model we consider is
	\begin{equation}\label{equation_general_model}
		Y_{it}=L_{it}^*
		+\sum_{p=1}^P \sum_{q=1}^Q X_{ip} H_{pq}^* Z_{qt}
		+\gamma_i^*+\delta_t^*+V_{it}^\top\beta^* +\varepsilon_{it}\,.
	\end{equation}
	the $\varepsilon_{it}$ is random noise. We are interested in estimating the unknown parameters $\bl^*$, $\bh^*$, $\gamma^*$, $\delta^*$ and $\beta^*$. This model allows for traditional econometric fixed effects for the units (the $\gamma_i^*$) and time effects (the $\delta_t^*$). It also allows for fixed covariate (these have time varying coefficients) and time covariates (with individual coefficients) and time varying individual covariates.
	Note that although we can subsume the unit and time fixed effects into the matrix $\bl^*$, we do not do so because we regularize the estimates of $\bl^*$, but do not wish to regularize the estimates of the fixed effects.
	
	The model can be rewritten as
	\begin{equation}
		\by=\bl^*+\bx \bh^*\bz^\top+
		\Gamma^*\id_T^\top+\id_N(\Delta^*)^\top
		+\left[V_{it}^\top\beta^*\right]_{it}+\bepsilon\,.
		\label{eq:model}
	\end{equation}
	Here $\bl^*$ is in $\reals^{N\times T}$, $\bh^*$ is in $\reals^{P\times Q}$, $\Gamma^*$ is in $\reals^{N\times 1}$ and $\Delta^*$ is in $\reals^{T\times 1}$. An slightly richer version of this model that allows linear terms in covariates can be defined as by
	\begin{equation}
		\by=\bl^*+\tbx \tbH^*\tbz^\top+
		\Gamma^*\id_T^\top+\id_N(\Delta^*)^\top
		+\left[V_{it}^\top\beta^*\right]_{it}+\bepsilon
		\label{eq:model_0icher}
	\end{equation}
	where $\tbx=[\bx|\bi_{N\times N}]$, $\tbz=[\bz|\bi_{T\times T}]$, and
	\[
	\tbH^* = \left[\begin{array}{cc}
	\bh_{X,Z}^* & \bh_{X}^*\\
	\bh_{Z}^* & {\bf 0}
	\end{array}\right]
	\]
	where $\bh_{XZ}^*\in\reals^{P\times Q}$, $\bh_{Z}^*\in\reals^{N\times Q}$, and $\bh_{X}^*\in\reals^{P\times T}$.  In particular,
	\begin{equation}
		\by=\bl^*+\tbx \tbH_{X,Z}^*\tbz^\top+
		\tbH_{Z}^*\tbz^\top+
		\bx\tbH_{X}^*+
		\Gamma^*\id_T^\top+\id_N(\Delta^*)^\top
		+\left[V_{it}^\top\beta^*\right]_{it}+\bepsilon
		\label{eq:model_0icher2}
	\end{equation}
	From now on, we will use the richer model \eqref{eq:model_0icher2} but abuse the notation and use notation $\bx,\bh^*,\bz$ instead of $\tbx,\tbH^*,\tbz$. Therefore, the matrix $\bh^*$ will be in $\reals^{(N+P)\times(T+Q)}$.
	
	We estimate $\bh^*$, $\bl^*$, $\delta^*$, $\gamma^*$, and $\beta^*$ by solving the following convex program,
	\begin{align*}
		\min_{\bh,\bl,\delta,\gamma,\beta}\left[
		\sum_{(i,t) \in \calo} \frac{1}{|\calo|} \left(Y_{it} -
		L_{it}-\sum_{p=1}^P \sum_{q=1}^Q X_{ip} H_{pq} Z_{qt}-\gamma_i-\delta_t-V_{it}\beta \right)^2
		+\lambda_L \|\bl\|_* + \lambda_H \|\bh\|_{1,e}\right]\,.
	\end{align*}
	Here $\|\bh\|_{1,e}=\sum_{i,t} |H_{it}|$ is the element-wise $\ell_1$ norm.
	We choose $\lambda_L$ and $\lambda_H$ through cross-validation.
	
	Solving this convex program is similar to the covariate-free case. In particular, by using a similar operator to $\shrink_\lambda$, defined in \S \ref{section_setup}, that performs coordinate descent with respect to $\bh$. Then we can apply this operator after each step of using $\shrink_\lambda$. Coordinate descent with respect to $\gamma$, $\delta$, and $\beta$ is performed similarly but using a simpler operation since the function is smooth with respect to them.
	
	\subsection{Leveraging Data From Treated Units}\label{subsec:impute-y0-y1}
	
	In previous sections we only focused on imputing $\by(0)$ to solve the treatment effect estimation problem. We note that this approach allows for very general assumptions on the treatment effect. For example if treatment effect has no (low-dimensional) patterns, imputing $\by(0)$ is the best one can do because $\by(1)$ would not have any pattern that can be used for imputation. We also note that in many of the applications there are very few treated unit/periods, so imputing the missing entries in $\by(1)$ would be much more challenging in practice.
	
	However, when the treatment effect is constant or has a low-rank pattern we can extend our approach and leverage the additional data from $\by(1)$. We describe these next.
	\begin{itemize}
		\item[(a)] {\bf When treatment effect is constant.} If the treatment effect is constant for every pair $(i,t)$, then we can consider the following natural extension of our estimator \eqref{objective_function1}.
		\begin{equation}\label{objective-function-constant-treatment}
			(\hat\bl,\hat\Gamma,\hat\Delta,\hat\tau)=
			\arg\min_{\bl,\Gamma,\Delta,\tau}\left\{\frac{1}{NT}\|\by-\bl-\Gamma\id_T^\top-\id_N\Delta^\top-\tau\bw\|^2_F+\lambda\|\bl\|_*\right\}
			\,,
		\end{equation}
		where variable $\tau\in\reals$ is used for estimating the constant treatment effect. Also, recall that $\bw$ is the binary treatment matrix. Note that here the squared error term includes all entries $(i,t)\in[N]\times[T]$.
		
		\item[(b)] {\bf When treatment effect has a low-rank pattern.} Assume the treatment effect is not constant but is such that the matrix $\by(1)$ has a low-rank expectation. Then we can impute $\by(1)$ the same way we impute $\by(0)$, using our estimator \eqref{objective_function1} applied to treated entries. Then we can use imputed matrix $\hat\by(0)$ and $\hat\by(1)$ to estimate the treatment effect matrix $\by(1)-\by(0)$.
		
	\end{itemize}

	\subsection{Autocorrelated Errors}
	
	One drawback of MC-NNM is that it does not take into account the time series nature of the observations. It is likely that the  $\bepsilon_{it}$ are correlated over time. We can take this into account by modifying the objective function. Let us consider this in the case without covariates, and, for illustrative purposes, let us use an autoregressive model of order one. Let $\by_{i\cdot}$ and $\bl_{i\cdot}$  be the $i^{th}$ row of $\by$ and $\bl$ respectively. The original objective function for $\calo=[N]\times[T]$ is
	\[ \frac{1}{|\calo|}\sum_{i=1}^N\sum_{t=1}^T (Y_{it}-L_{it})^2+\lambda_L \|\bl\|_*
	=\frac{1}{|\calo|} \sum_{i=1}^N (Y_{i\cdot}-L_{i\cdot})(Y_{i\cdot}-L_{i\cdot})^\top+\lambda_L \|\bl\|_*
	.\]
	We can modify this to $\sum_{i=1}^N (Y_{i\cdot}-L_{i\cdot})\bOmega^{-1}(Y_{i\cdot}-L_{i\cdot})^\top/|\calo|+\lambda_L \|\bl\|_*$,
	where the choice for the $T\times T$ matrix $\bOmega$ would reflect the autocorrelation in the $\bepsilon_{it}$. For example, with a first order autoregressive process, we would use $\Omega_{ts}=\sigma^2\rho^{|t-s|}$, with $\rho$ an estimate of the autoregressive coefficient. Similarly, for the more general version $\calo\subset[N]\times[T]$, we can use the function
	\[
	\frac{1}{|\calo|} \sum_{(i,t)\in\calo}\sum_{(i,s)\in\calo} (Y_{it}-L_{it})[\bOmega^{-1}]_{ts}(Y_{is}-L_{is})+\lambda_L \|\bl\|_*\,.
	\]
	
	\subsection{Weighted Loss Function}\label{subsec:weightedLossFunction}
	
	Another limitation of MC-NNM is that it puts equal weight on all observed elements of the difference $\by-\bl$ (ignoring the covariates). Ultimately we care solely about predictions of the model for the missing elements of $\by$, and for that reason it is natural to emphasize the fit of the model for elements of $\by$ that are observed, but that are similar to the elements that are missing. In the program evaluation literature this is often achieved by weighting the fit by the propensity score, the probability of outcomes for a unit
	being missing.
	
	We can do so in the current setting by modelling this probability in terms of the covariates and a latent factor structure.
	Let the propensity score be $\ce_{it}=\prob(W_{it}=1|X_i,Z_t,V_{it})$, and let
	$\bce$ be the $N\times T$ matrix with typical element $\ce_{it}$. Let
	us again consider the case without covariates. In that case we may wish to model the assignment $\bw$ as
	\[ \bw_{N\times T}=\bce_{N\times T}+\bbeta_{N\times T}.\]
	We can estimate this using the same matrix completion methods as before, now without any missing values:
	\[\hat{\bce}=\arg\min_{\bce}\frac{1}{N T}
	\sum_{(i,t)} \left(W_{it} -
	\ce_{it} \right)^2+\lambda_L \|\bce\|_*\,.
	\]
	Given the estimated propensity score we can then weight the objective function for estimating $\bl^*$:
	\[\hat{\bl}=\arg\min_{\bl} \fncalo
	\sum_{(i,t) \in \calo} \frac{\hat \ce_{it}}{1-\hat \ce_{it}}\left(Y_{it} -
	L_{it} \right)^2+\lambda_L \|\bl\|_*\,.
	\]

	\subsection{Relaxing the Dependence of Theorem \ref{thm:consistency} on $\pc$}
	
	Recall from \S \ref{notation} that the average number of control units is $\sum_{i=1}^N\pi_T^{(i)}$. Therefore, the fraction of control units is $\sum_{i=1}^N\pi_T^{(i)}/N$. However, the estimation error in Theorem \ref{thm:consistency} depends on $\pc=\min_{1\le i\le N}\pi_T^{(i)}$ rather than $\sum_{i=1}^N\pi_T^{(i)}/N$. The reason for this, as discussed in \S \ref{notation} is due to special classes of matrices $\bl^*$ where most of the rows are nearly zero (e.g, when only one row is non-zero). In order to relax this constraint we would need to restrict the family of matrices $\bl^*$. An example of such restriction is given by \citet{negahban2012restricted} where they assume $\bl^*$ is not too spiky. Formally, they assume the ratio $\|\bl^*\|_\mmax/\|\bl^*\|_F$ should be of order $1/\sqrt{NT}$ up to logarithmic terms. To see the intuition for this, in a matrix with all equal entries this ratio is $1/\sqrt{NT}$ whereas in a matrix where only the $(1,1)$ entry is non-zero the ratio is $1$. While both matrices have rank $1$, in the former matrix the value of $\|\bl^*\|_F$ is obtained from most of the entries. In such situations, one can extend our results and obtain an upper bound that depends on $\sum_{i=1}^N\pi_T^{(i)}/N$.
	
	\subsection{Nearly Low-rank Matrices}\label{subsec:nearlyLowRank}
	
	Another possible extension of Theorem \ref{thm:consistency} is to the cases where $\bl^*$ may have high rank, but most of its singular values are small. More formally, if $\sigma_1\ge \cdots > \sigma_{\min(N,T)}$ are singular values of $\bl^*$, one can obtain upper bounds that depend on $k$ and $\sum_{r=k+1}^{\min(N,T)}\sigma_{r}$ for any $k\in[\min(N,T)]$. One can then optimize the upper bound by selecting the best $k$. In the low-rank case such optimization leads to selecting $k$ equal to $R$. This type of more general upper bound has been proved in some of prior matrix completion literature, e.g. \citet{negahban2012restricted}. We expect their analyses would be generalize-able to our setting (when entries of $\calo$ are not independent).
	
	\subsection{Additional Missing Entries}\label{subsec:additionalMissing}
	
	In \S \ref{notation} we assumed that all entries $(i,t)$ of $\by$ for $t\le t_i$ are observed. However, it may be possible that some such values are missing due to lack of data collection. This does not mean that any treatment occurred in the pre-treatment period. Rather, such scenario can occur when measuring outcome values is costly. In this case, one can extend Theorem \ref{thm:consistency} to the setting with $
	\calo = \left[\bigcup_{i=1}^N \Big\{(i,1),(i,2),\ldots,(i,t_i)\Big\}\right]\setminus \calo_{\text{miss}}$,
	where each $(i,t)\in \cup_{i=1}^N \{(i,1),(i,2),\ldots,(i,t_i)\}$ can be in $\calo_{\text{miss}}$, independently, with probability $p$ for $p$ that is not too large.
	
	%

	%
	
	%
	\section{Conclusions}
	
	We  present new results  for estimation of causal effects in panel or longitudinal data settings. The proposed estimator, building on the interactive fixed effects and matrix completion literatures has attractive  computational properties in settings with large $N$ and $T$, and allows for a relatively large number of factors. We show how this set up relates to the program evaluation and synthetic control literatures.
	In illustrations we show that the method adapts well to different configurations of the data, and find that generally it outperforms the synthetic control estimators proposed  \cite{abadie2010synthetic} and the elastic net estimators proposed by \cite{doudchenko}.

\bibliography{references}
\bibliographystyle{plainnat}
	
\newpage
\appendix

	%

	\section{Online Appendix for `` Matrix Completion Methods for Causal Panel Data Models'': Proofs}\label{sec:supp_theorey}

	\subsection{Proof of Theorem \ref{thm:representation2}}\label{subsec:representation2Proof}
	
	To prove part $(i)$, we first state that
	if
	\[ \hat\bl=\argmin_{\substack{\bl: \|\bl\|_{\mmax}\le \lmax}}\left \{\frac{1}{\ncalo}{\|\bp_\calo(\by-\bl)\|^2_F}
	+\lambda\|\bl\|_*\right\}\,,\]
	and
	\[(\hat R,\hat\ba,\hat\bb)=
	\arg\min_R\arg\min_{\ba\in\mmm^{N,R},\bb\in\mmm^{T,R}}
	Q(\by;R,\ba,\bb,\gamma,\delta)
	+\frac{\lambda}{2}\|\ba\|^2_F+\frac{\lambda}{2}\|\bb\|^2_F,\]
	then
	\begin{equation}\label{meen} \hat\bl=\hat\ba\hat\bb^\top.\end{equation}
	This follows from the fact (Lemma 6, \citet{mazumder2010spectral}) that
	\[ \left\|\bl\right\|_*=\min_{\ba,\bb:\bl=\ba\bb^\top} \frac{1}{2}\left(
	\left\|\ba\right\|^2_F+\left\|\ba\right\|^2_F
	\right).\]
	Now define
	\[ \hat \bl(\gamma,\delta)=\argmin_{\substack{\bl: \|\bl\|_{\mmax}\le \lmax}}\left \{\frac{1}{\ncalo}{\|\bp_\calo(\by-\bl-\gamma\id^\top_T-\id_N\delta^\top)\|^2_F}
	+\lambda\|\bl\|_*\right\}\,,\]
	and
	\[(\hat R(\gamma,\delta),\hat\ba(\gamma,\delta),\hat\bb(\gamma,\delta))=\]
	\[
	\arg\min_R\arg\min_{\ba\in\mmm^{N,R},\bb\in\mmm^{T,R}}
	Q(\by;R,\ba,\bb,\gamma,\delta)
	+\frac{\lambda}{2}\|\ba\|^2_F+\frac{\lambda}{2}\|\bb\|^2_F\]
	\[
	=\arg\min_R\arg\min_{\ba\in\mmm^{N,R},\bb\in\mmm^{T,R}}
	Q(\by-\gamma\id^\top_T-\id_N\delta^\top;R,\ba,0,0)
	+\frac{\lambda}{2}\|\ba\|^2_F+\frac{\lambda}{2}\|\bb\|^2_F,\]
	which, by (\ref{meen}) implies
	\[ \bl(\gamma,\delta)=\ba(\gamma,\delta)\bb(\gamma,\delta)^\top).\]

	Define
	\[ S(\by;\bl,\gamma,\delta)=
	\frac{1}{\ncalo}{\|\bp_\calo(\by-\bl-\gamma\id^\top_T-\id_N\delta^\top)\|^2_F}
	+\lambda\|\bl\|_*.\]
	Also define
	\[ (\hat \bl,\hat\gamma,\hat\delta)=\argmin_{\substack{\bl: \|\bl\|_{\mmax}\le \lmax},\gamma,\delta}\left \{\frac{1}{\ncalo}{\|\bp_\calo(\by-\bl-\gamma\id^\top_T-\id_N\delta^\top)\|^2_F}
	+\lambda\|\bl\|_*\right\}\,,\]
	\[(\tilde R,\tilde\ba,\tilde\bb,\tilde\gamma,\tilde\delta)=\]
	\[
	\arg\min_R\arg\min_{\ba\in\mmm^{N,R},\bb\in\mmm^{T,R},\gamma,\delta}
	Q(\by;R,\ba,\bb,\gamma,\delta)
	+\frac{\lambda}{2}\|\ba\|^2_F+\frac{\lambda}{2}\|\bb\|^2_F,\]
	\[\hat Q=\frac{1}{\ncalo}{\|\bp_\calo(\by-\hat\bl-\hat\gamma\id^\top_T-\id_N\hat\delta^\top)\|^2_F}
	+\lambda\|\bl\|_*,
	\]
	and
	\[ \tilde Q=Q(\by;\tilde R,\tilde \ba,\tilde \bb,\tilde \gamma,\tilde\delta).
	\]
	In order to prove that $\hat\bl=\tilde\ba\tilde\bb^\top$, we prove that
	$(\hat\gamma,\hat\delta)=(\tilde\gamma,\tilde\delta)$.
	
	Suppose $(\hat\gamma,\hat\delta)\neq (\tilde\gamma,\tilde\delta)$. Then
	\[ \tilde Q=Q(\by;\tilde R,\tilde \ba,\tilde \bb,\tilde \gamma,\tilde\delta)
	< Q(\by;\hat R(\hat\gamma,\hat\delta),\hat\ba(\hat\gamma,\hat\delta),\hat\bb(\hat\gamma,\hat\delta),\hat \gamma,\hat\delta)=\hat Q.\]
	But, also
	\[ \hat Q=S(\by;\hat\bl,\hat\gamma,\hat\delta)<
	S(\by;\hat\bl(\tilde\gamma,\tilde\delta),\tilde\gamma,\tilde\delta)=Q(\by;\tilde R,\tilde \ba,\tilde \bb,\tilde \gamma,\tilde\delta),\]
	which leads to a contradiction. Hence $(\hat\gamma,\hat\delta)=(\tilde\gamma,\tilde\delta)$.
	
	Next, consider part $(ii)$. Consider minimizing
	\begin{equation}\label{peen}\min_R\min_{\ba\in\mmm^{N,T-1},\bb\in\mmm^{T,T-1},\gamma,\delta}
		Q(\by;R,\ba,\bb,\gamma,\delta)
		,\end{equation}
	subject to
	\[
	R=T-1,\hskip0.5cm
	\ba=
	\left(
	\begin{array}{c}
	\ty \\
	\bytt
	\end{array}\right),\hskip0.5cm \gamma=0, \hskip0.5cm \delta_1=\delta_2=\ldots=\delta_{T-1}=0.
	\]
	Partition $\bb$ into
	\[ \bb=\left(\begin{array}{c}
	\bb_0\\ \bbb_2^\top
	\end{array}\right).\]
	Substituting for the $Q(\cdot)$ and for the restricted parameters, the minimization problem in (\ref{peen}) is identical to
	\[ \min_{\bb_0\in\mmm^{T-1,T-1},\bbb_2\in\mmm^{T-1,1},\delta_T}
	\left\|P_\calo \left(
	\left(
	\begin{array}{cc}
	\ty & \bye\\
	\bytt & ?
	\end{array}\right)-
	\left(
	\begin{array}{c}
	\ty \\
	\bytt
	\end{array}\right)
	\left(\begin{array}{c}
	\bb_0\\ \bbb_2^\top
	\end{array}\right)^\top
	-\id_N\left(\begin{array}{ccccc} 0 & 0 & \ldots & 0 & \delta_T\end{array}\right)\right)\right\|^2_F
	\]
	\[= \min_{\bb_0\in\mmm^{T-1,T-1},\bbb_2\in\mmm^{T-1,1},\delta_T}
	\left\|P_\calo \left(
	\left(
	\begin{array}{cc}
	\ty -\ty \bb_0& \bye-\ty\bbb^\top_2-\id_{N-1}\delta_T\\
	\bytt -\bytt\bb_0& ?
	\end{array}\right)\right)\right\|^2_F
	\]
	\[
	=\min_{\bb_0\in\mmm^{T-1,T-1},\bbb_2\in\mmm^{T-1,1},\delta_T}
	\left\{
	\left\|
	\ty -\ty \bb_0\right\|^2_F
	+
	\left\|\bytt -\bytt\bb_0
	\right\|^2_F
	+
	\left\|\bye-\ty\bbb^\top_2-\id_{N-1}\delta_T
	\right\|^2_F
	\right\}
	\]
	\[=\min_{\bb_0\in\mmm^{T-1,T-1}}
	\left\{
	\left\|
	\ty -\ty \bb_0\right\|^2_F
	+
	\left\|\bytt -\bytt\bb_0
	\right\|^2_F
	\right\}
	+\min_{\bbb_2\in\mmm^{T-1,1},\delta_T}
	\left\|\bye-\ty\bbb^\top_2-\id_{N-1}\delta_T
	\right\|^2_F.
	\]
	The solution for $\bb_0$ is the $(T-1)\times (T-1)$ dimensional identity matrix. The solution for $\bb_2$ and $\delta_T$ are the solution to the regression of $\bye$ on a constant and $\ty$, which proves the second part.
	
	The remaining parts follow the same argument.
	$\square$

	\subsection{Proof of Theorem \ref{thm:consistency}}\label{subsec:consistencyProof}
	
	First, we will discuss three main steps that are needed for the proof.
	
	\paragraph{Step 1:} We show an upper bound for the sum of squared errors for all $(i,t)\in\calo$ in terms of the regularization parameter $\lambda$, rank of $\bl^*$, $\|\bl^*-\hat\bl\|_F$, and $\|\be\|_\op$ where $\be\equiv \sum_{(i,t)\in\calo}\err_{it}\ba_{it}$.
	\begin{lemma}[Adapted from \citet{negahban2011estimation}]\label{lem:error_bound}
		Then for all $\lambda\ge 3\|\be\|_\op/|\calo|$,
		\begin{equation}\label{eq:error_bound}
			\sum_{(i,t)\in\calo}\frac{\langle\ba_{it},\bl^*-\hat\bl\rangle^2}{|\calo|} \leq 10 \lambda\sqrt{R}\, \|\bl^*-\hat\bl\|_F\,.
		\end{equation}
	\end{lemma}
	This type of result has been shown before by \citet{recht2011simpler,negahban2011estimation,koltchinskii2011nuclear,klopp2014noisy}. For convenience of the reader, we include its proof in \S \ref{sec:supp_theorey}. Similar results also appear in the analysis of LASSO type estimators (for example see \cite{buhlmann2011statistics} and references therein).
	\begin{remark}\label{rem:random-O}
		We also note that, while $\calo$ is a random variable, Lemma \ref{lem:error_bound} makes a deterministic statement. Specifically, its result holds for any realization of the random variable $\calo$. In fact, its proof is based on linear algebra facts and is not using any probabilistic argument.
	\end{remark}

	\paragraph{Step 2:} The upper bound provided by Lemma \ref{lem:error_bound} contains $\lambda$ and also requires the condition $\lambda\ge 3\|\be\|_\op/|\calo|$. Therefore, in order to have a tight bound, it is important to show an upper bound for $\|\be\|_\op$ that holds with high probability. Next lemma provides one such result.
	\begin{lemma}\label{lem:bound_err_matrix}
		There exist a constant $C_1$ such that
		\[
		\|\be\|_\op \le C_1\sigma\max\left[\sqrt{N\log(N+T)},\sqrt{T}\log^{3/2}(N+T)\right]\,,
		\]
		with probability greater than $1-(N+T)^{-2}$.
	\end{lemma}
	This result uses a concentration inequality for sum of random matrices to find a bound for $\|\be\|_\op$. We note that previous papers,  \citep{recht2011simpler,negahban2011estimation,koltchinskii2011nuclear,klopp2014noisy}, contain a similar step but in their case $\calo$ is obtained by independently sampling elements of $[N]\times[T]$. However, in our case observations from each row of the matrix are correlated. Therefore, prior results do not apply. In fact, the correlation structure deteriorates the type of upper bound that can be obtained for $\|\be\|_\op$.
	
	\paragraph{Step 3:} The last main step is to show that, with high probability, the random variable on the left hand side of \eqref{eq:error_bound} is larger than a constant fraction of $\|\hat\bl-\bl^*\|_F^2$ up to an additive term. In high-dimensional statistics literature this property is also referred to as \emph{Restricted Strong Convexity}, \citep{negahban2012unified,negahban2011estimation,negahban2012restricted}. The following Lemma states this property for our setting and its proof is similar to the proof of Theorem 1 in \citep{negahban2012restricted}, Lemma 12 in \citep{klopp2014noisy}, or Corollary 3.1 in \citep{hamidi2019lowrank} for the cases that observation process $\calo$ does not have a dependency structure, like in our setting. Therefore, for completeness we  provide a summary of this proof (adapted to our setting) in \S \ref{subsec:pf-lem-lower-bound-err}.
	%
	%
	%
	\begin{lemma}\label{lem:lower_bound_err}
If the estimator $\hat\bl$ satisfies $\|\hat\bl-\bl^*\|_F^2\ge4\lmax^2\theta/\pc$ for a positive number $\theta$, then there is a constant $C\ge 0.001$ such that, when $C\theta>T$, we have
\[
\prob_\pi\left\{\frac{\pc}{2}\,\|\hat\bl-\bl^*\|_F^2 > \sum_{(i,t)\in\calo}\langle\ba_{it},\hat\bl-\bl^*\rangle^2+8\lmax^2T\sqrt{N}\right\}
\le 2\exp\left(-\frac{C\theta}{T}\right)\,.
\]
\end{lemma}
	Now we are ready to prove the main theorem.
	\begin{proof}[Proof of Theorem \ref{thm:consistency}]
		Let $\bdelta = \bl^*-\hat\bl$. Then using Lemma \ref{lem:bound_err_matrix} and selecting $\lambda$ equal to $3\|\be\|_\op/|\calo|$ in Lemma \ref{lem:error_bound}, with probability greater than $1-(N+T)^{-2}$, we have
		\begin{align}
			\sum_{(i,t)\in\calo}\frac{\langle\ba_{it},\bdelta\rangle^2}{|\calo|} &\leq
			\frac{30 C_1\sigma\sqrt{R} \left[\sqrt{N\,\log(N+T)}\vee\sqrt{T}\,\log^{3/2}(N+T)\right]}{|\calo|}\,\, \|\bdelta\|_F\,.
			\label{eq:lem2_plus_lem3}
		\end{align}
		Now, we use Lemma \ref{lem:lower_bound_err} to find a lower bound for the left hand side of \eqref{eq:lem2_plus_lem3}. But first note that if $\pc\|\bdelta\|_F^2\le C'\lmax^2T\log(N+T)$, for a constant $C'$,
		then
		\begin{align*}
		\frac{\|\bdelta\|_F}{\sqrt{NT}} &\le \sqrt{ \frac{C'\lmax^2\log(N+T)}{N\,\pc}}
		\end{align*}
		holds which proves Theorem \ref{thm:consistency}. Otherwise, for $\theta\equiv C'' T\log(N+T)$ with a large enough constant $C''$ such that the condition $0.001\times\theta>T$ is satisfied, we have
		\begin{align*}
	\|\bdelta\|_F^2 &\ge  \frac{4\lmax^2\theta}{\pc}\,.
		\end{align*}
		Therefore, we can invoke Lemma \ref{lem:lower_bound_err} and obtain
		\begin{align}
		\prob_\pi\left\{\frac{\pc}{2}\,\|\bdelta\|_F^2 \le \sum_{(i,t)\in\calo}\langle\ba_{it},\bdelta\rangle^2+8\lmax^2T\sqrt{N}\right\}
			\ge 1-\frac{1}{(N+T)^2}\,.\label{eq:lemma3-applied}
		\end{align}
		Combining \eqref{eq:lemma3-applied}, \eqref{eq:lem2_plus_lem3}, and the union bound we have, with probability greater than $1-2(N+T)^{-2}$,
			\begin{align*}
		\frac{\pc\|\bdelta\|_F^2}{2}&\le C'''\sigma\sqrt{R} \left[\sqrt{N\,\log(N+T)}\vee\sqrt{T}\,\log^{3/2}(N+T)\right]\|\bdelta\|_F+8\lmax^2T\sqrt{N}\\
		&\le \frac{4C'''\sigma^2R \left[N\,\log(N+T)\vee T\,\log^{3}(N+T)\right]}{\pc}+\frac{\pc\|\bdelta\|_F^2}{4}+8\lmax^2T\sqrt{N}\,,
		\end{align*}
		where the last step uses the inequality $2ab\le a^2+b^2$. Therefore, we obtain
			\begin{align*}
		\|\bdelta\|_F^2&\le \frac{16C'''\sigma^2R \left[N\,\log(N+T)\vee T\,\log^{3}(N+T)\right]}{\pc^2}+\frac{32\lmax^2T\sqrt{N}}{\pc}
		\end{align*}
		The main result now follows after dividing both sides with $\sqrt{NT}\|\bdelta\|_F$, using $T\ge C_0\log(N+T)$, and choosing a large enough constant $C$ in Eq. \eqref{eq:consistency}.
	\end{proof}

	\subsection{Proof of Lemma \ref{lem:error_bound}}
	
	Variants of this Lemma for similar models have been proved before. But for completeness we include its proof that is adapted from \cite{negahban2011estimation}.
	\begin{proof}[Proof of Lemma \ref{lem:error_bound}]
		Let
		\[
		f(\bl)\equiv \sum_{(i,t)\in\calo} \frac{\left(Y_{it}-L_{it}\right)^2}{|\calo|}+\lambda\|\bl\|_*\,.
		\]
		Now, using the definition of $\hat\bl$,
		\[
		f(\hat{\bl})\le f(\bl^*)\,,
		\]
		which is equivalent to
		\begin{equation}
			\sum_{(i,t)\in\calo} \frac{\langle \bl^*-\hat{\bl},\ba_{it}\rangle^2}{|\calo|}+
			2 \sum_{(i,t)\in\calo}\frac{\err_{it}\langle \bl^*-\hat{\bl},\ba_{it}\rangle}{|\calo|}
			+
			\lambda\|\hat{\bl}\|_*
			\le
			\lambda\|\bl^*\|_*\,.
		\end{equation}
		Now, using the definition of $\be$ and $\bdelta$, the above equation gives
		\begin{align}
			\sum_{(i,t)\in\calo} \frac{\langle \bdelta,\ba_{it}\rangle^2}{|\calo|}
			&\le
			-\frac{2}{|\calo|}\langle \bdelta,\be\rangle+\lambda\|\bl^*\|_*-\lambda\|\hat{\bl}\|_*\\
			&\stackrel{(a)}{\le} \frac{2}{|\calo|}\|\bdelta\|_*\|\be\|_\op+\lambda\|\bl^*\|_*-\lambda\|\hat{\bl}\|_*\label{eq:2}\\
			&\le \frac{2}{|\calo|}\|\bdelta\|_*\|\be\|_\op+\lambda\|\bdelta\|_*\\
			&\stackrel{(b)}{\le} \frac{5}{3}\lambda\|\bdelta\|_*\label{eq:le-2-lambda-nuc-norm-delta}\,.
		\end{align}
		Here, $(a)$ uses inequality $|\langle\ba,\bb\rangle|\le \|\ba\|_\op\|\bb\|_{\mmax}$ which is due to the fact that operator norm is the dual of nuclear norm,
		and $(b)$ uses the assumption $\lambda\ge3\|\be\|_\op/|\calo|$.
		Before continuing with the proof of Lemma \ref{lem:error_bound} we state the following Lemma that is proved later in this section.
		\begin{lemma}\label{lem:error_decomposition}
			Let $\bdelta\equiv \bl^*-\hat{\bl}$ for $\lambda\ge3\|\be\|_\op/|\calo|$ Then there exist a decomposition $\bdelta=\bdelta_1+\bdelta_2$ such that
			\begin{enumerate}
				\item[(i)] $\langle\bdelta_1,\bdelta_2\rangle=0$.
				\item[(ii)] $\rank(\bdelta_1)\le 2R$.
				\item[(iii)] $\|\bdelta_2\|_*\le 3 \|\bdelta_1\|_*$.
			\end{enumerate}
		\end{lemma}
		Now, invoking the decomposition $\bdelta=\bdelta_1+\bdelta_2$ from Lemma \ref{lem:error_decomposition} and using the triangle inequality, we obtain
		\begin{align}
			\|\bdelta\|_*
			\stackrel{(c)}{\le} 4 \|\bdelta_1\|_*
			\stackrel{(d)}{\le} 4\sqrt{2R} \|\bdelta_1\|_F
			\stackrel{(e)}{\le}  4\sqrt{2R} \|\bdelta\|_F\,.
			\label{eq:low-rankness}
		\end{align}
		where $(c)$ uses Lemma \ref{lem:error_decomposition}$(iii)$, $(d)$ uses Lemma \ref{lem:error_decomposition}$(ii)$ and Cauchy-Schwarz inequality, and $(e)$ uses Lemma \ref{lem:error_decomposition}$(i)$. Combining this with \eqref{eq:le-2-lambda-nuc-norm-delta} we obtain
		\begin{align}
			\sum_{(i,t)\in\calo} \frac{\langle \bdelta,\ba_{it}\rangle^2}{|\calo|}
			&\le
			10 \lambda\sqrt{R}\, \|\bdelta\|_F\,,
		\end{align}
		which finishes the proof of Lemma \ref{lem:error_bound}.
	\end{proof}
	\begin{proof}[Proof of Lemma \ref{lem:error_decomposition}]
		Let $\bl^*=\bu_{N\times R}\bs_{R\times R}(\bv_{T\times R})^\top$ be the singular value decomposition for the rank $R$ matrix $\bl^*$. Let $\bp_U$ be the projection operator onto column space of $\bu$ and let $\bp_{U^\perp}$ be the projection operator onto the orthogonal complement of the column space of $\bu$. Let us recall a few linear algebra facts about these projection operators. If columns of $\bu$ are denoted by $u_1,\ldots,u_R$, since $\bu$ is unitary, $\bp_{U}=\sum_{i=1}^R u_iu_i^\top$. Similarly, $\bp_{U^\perp}=\sum_{i=R+1}^Nu_iu_i^\top$ where $u_1,\ldots,u_0,u_{R+1},\ldots,u_N$ forms an orthonormal basis for $\reals^N$. In addition, the projector operators are idempotent (i.e., $\bp_{U}^2 = \bp_{U}, \bp_{U^\perp}^2 = \bp_{U^\perp}$), $\bp_{U}+\bp_{U^\perp}=\bi_{N\times N}$.
		
		Define $\bp_V$ and $\bp_{V^\perp}$ similarly. Now, we define $\bdelta_1$ and $\bdelta_2$ as follows:
		\[
		\bdelta_2\equiv \bp_{U^\perp}\bdelta \bp_{V^\perp}~~~,~~~~\bdelta_1\equiv\bdelta - \bdelta_2\,.
		\]
		It is easy to see that
		\begin{align}
			\bdelta_1&=(\bp_{U}+\bp_{U^\perp})\bdelta(\bp_{V}+\bp_{V^\perp}) - \bp_{U^\perp}\bdelta \bp_{V^\perp}\\
			&=\bp_{U}\bdelta + \bp_{U^\perp}\bdelta \bp_{V}\,.\label{eq:delta1_decomp}
		\end{align}
		Using this fact we have
		\begin{align}
			\langle\bdelta_1,\bdelta_2\rangle&= \tr
			\left(
			\bdelta^\top\bp_{U}\bp_{U^\perp}\bdelta \bp_{V^\perp}+
			\bp_{V}\bdelta^\top \bp_{U^\perp}\bp_{U^\perp}\bdelta \bp_{V^\perp}
			\right)\\
			&= \tr
			\left(
			\bp_{V}\bdelta^\top \bp_{U^\perp}\bdelta \bp_{V^\perp}
			\right)\\
			&= \tr
			\left(
			\bdelta^\top \bp_{U^\perp}\bdelta \bp_{V^\perp}\bp_{V}
			\right)=0
		\end{align}
		that gives part (i). Note that we used $\tr(\ba\bb)=\tr(\bb\ba)$.
		
		Looking at \eqref{eq:delta1_decomp}, part (ii) also follows since both $\bp_{U}$ and $\bp_{V}$ have rank $r$ and sum of two rank $r$ matrices has rank at most $2r$.
		
		Before moving to part (iii), we note another property of the above decomposition of $\bdelta$ that will be needed next. Since the two matrices $\bl^*$ and $\bdelta_2$ have orthogonal singular vectors to each other,
		\begin{align}
			\|\bl^*+\bdelta_2\|_* =  \|\bl^*\|_*+\|\bdelta_2\|_*\,. \label{eq:nucnorm_decomp_L*+Delta_2}
		\end{align}

		On the other hand, using inequality \eqref{eq:2}, for $\lambda\ge 3\|\be\|_\op/|\calo|$ we have
		\begin{align}
			\lambda\left(\|\hat{\bl}\|_* - \|\bl^*\|_*\right)
			&\le \frac{2}{|\calo|}\|\bdelta\|_*\|\be\|_\op\nonumber\\
			&\le \frac{2}{3}\lambda\|\bdelta\|_*\nonumber\\
			&\le \frac{2}{3}\lambda\left(\|\bdelta_1\|_*+\|\bdelta_2\|_*\right)\,.\label{eq:3}
		\end{align}
		Now, we can use the following for the left hand side
		\begin{align*}
			\|\hat{\bl}\|_* - \|\bl^*\|_*
			&= \|\bl^*+\bdelta_1+\bdelta_2\|_* - \|\bl^*\|_*\\
			&\ge \|\bl^*+\bdelta_2\|_* -\|\bdelta_1\|_* - \|\bl^*\|_*\\
			&\stackrel{(f)}{=} \|\bl^*\|_*+\|\bdelta_2\|_* -\|\bdelta_1\|_* - \|\bl^*\|_*\\
			&= \|\bdelta_2\|_* -\|\bdelta_1\|_*\,.
		\end{align*}
		Here $(f)$ follows from \eqref{eq:nucnorm_decomp_L*+Delta_2}. Now, combining the last inequality with \eqref{eq:3} we get
		\begin{align*}
			\|\bdelta_2\|_* -\|\bdelta_1\|_*
			&\le \frac{2}{3}\left(\|\bdelta_1\|_*+\|\bdelta_2\|_*\right)\,.
		\end{align*}
		That finishes proof of part (iii).
	\end{proof}
	
	\subsection{Proof of Lemma \ref{lem:bound_err_matrix}}
	
	First we state the matrix version of Bernstein inequality for rectangular matrices (see \citet{tropp2012user} for a derivation of it).
	\begin{prop}[Matrix Bernstein Inequality]\label{prop:matrix_bernstein}
		Let $\bz_1,\ldots,\bz_N$ be independent matrices in $\reals^{d_1\times d_2}$ such that $\E[\bz_i]=\bzero$ and $\|\bz_i\|_\op\le D$ almost surely for all $i\in[N]$. Let $\sigma_Z$ be such that
		\[
		\sigma_{Z}^2\ge \max\left\{~\left\|\sum_{i=1}^N\E[\bz_i\bz_i^\top]\right\|_\op~,~\left\|\sum_{i=1}^N\E[\bz_i^\top\bz_i\right\|_\op~\right\}\,.
		\]
		Then, for any $\alpha\ge 0$
		\begin{equation}
			\prob\left\{~
			\left\|\sum_{i=1}^N\bz_i\right\|_\op \ge \alpha
			\right\}\le (d_1+d_2) \exp\left[\frac{-\alpha^2}{2\sigma_Z^2+(2D\alpha)/3}\right]\,.
		\end{equation}
	\end{prop}

	\begin{proof}[Proof of Lemma \ref{lem:bound_err_matrix}]
		%
%
Our goal is to use Proposition \ref{prop:matrix_bernstein}.
		Define the sequence of independent random matrices $\bb_1,\ldots,\bb_N$ as follows. For every $i\in[N]$, define
		\[
		\bb_i =\sum_{t=1}^{t_i} \err_{it}\ba_{it}\,.
		\]
		By definition, $\be=\sum_{i=1}^N\bb_i$ and $\E[\bb_i]=\bzero$ for all $i\in[N]$.
		Define the bound $D\equiv C_2\sigma\sqrt{\log(N+T)}$ for a large enough constant $C_2$. For each $(i,t)\in\calo$ define $\bar{\err}_{it}=\err_{it}\ind_{|\err_{it}|\leq D}$. Also define
		$\overline{\bb}_i=\sum_{t=1}^{t_i}\bar{\err}_{it}\ba_{it}$ for all $i\in[N]$.
		
		Using union bound and the fact that for $\sigma$-sub-Gaussian random variables $\err_{it}$ we have
		$\prob(|\err_{it}|\ge t)\le 2\exp\{-t^2/(2\sigma^2)\}$ gives, for each $\alpha\ge 0$,
		\begin{align}
			\prob\{~\|\be\|_\op \ge \alpha\}
			%
			&\le \prob\left\{~\left\|\sum_{i=1}^N\overline{\bb}_i\right\|_\op \ge \alpha\right\}+2NT\exp\left\{\frac{-D^2}{2\sigma^2}\right\}\nonumber\\
			&\le \prob\left\{~\left\|\sum_{i=1}^N\overline{\bb}_i\right\|_\op \ge \alpha\right\}+\frac{1}{(N+T)^3}\,.\label{eq:P(E)<P(sum_B)+err}
		\end{align}
		Now, for each $i\in[N]$, define $\bz_i\equiv \overline{\bb}_i-\E[\overline{\bb}_i]$. Then,
		\begin{align*}
			\left\|\sum_{i=1}^N\overline{\bb}_i\right\|_\op
			&\le
			\left\|\sum_{i=1}^N\bz_i\right\|_\op + \left\|\E\left[\sum_{1\le i\le N}\overline{\bb}_i\right]\right\|_\op\\
			&\le
			\left\|\sum_{i=1}^N\bz_i\right\|_\op + \left\|\E\left[\sum_{1\le i\le N}\overline{\bb}_i\right]\right\|_F
			&\le
			\left\|\sum_{i=1}^N\bz_i\right\|_\op + \sqrt{NT}\left\|\E\left[\sum_{1\le i\le N}\overline{\bb}_i\right]\right\|_{\mmax} \,.
		\end{align*}
		But since each $\err_{it}$ has mean zero,
		\begin{align*}
			\Big|\E\left[\bar{\err}_{it}\right]\Big|=\Big|\E\left[\err_{it}\ind_{|\err_{it}|\le D}\right]\Big|=\Big|\E\left[\err_{it}\ind_{|\err_{it}|\ge D}\right]\Big|
			&\le \sqrt{\E[\err_{it}^2]\,\prob(|\err_{it}|\ge D)}\\
			&\le \sqrt{2\sigma^2\,\exp[-D^2/(2\sigma^2)]}\\
			&\le \frac{\sigma}{(N+T)^4}\,.
		\end{align*}
		Therefore,
		\begin{align*}
			\sqrt{NT}\left\|\E\left[\sum_{1\le i\le N}\overline{\bb}_i\right]\right\|_{\mmax} \le \frac{\sigma\sqrt{NT}}{(N+T)^4}\le \frac{\sigma}{(N+T)^3}\,,
		\end{align*}
		which gives
		\begin{align}
			\left\|\sum_{i=1}^N\overline{\bb}_i\right\|_\op
			&\le
			\left\|\sum_{i=1}^N\bz_i\right\|_\op + \frac{\sigma}{(N+T)^3}\,.\label{eq:sumB<sumZ+err}
		\end{align}

		We also note that $\|\bz_i\|_\op\le 2D\sqrt{T}$ for all $i\in[N]$. The next step is to calculate $\sigma_Z$ defined in the Proposition \ref{prop:matrix_bernstein}. We have,
		\begin{align}
			\left\|\sum_{i=1}^N\E[\bz_i\bz_i^\top]\right\|_\op &\le \max_{(i,t)\in\calo}\left\{\E\left[(\bar{\err}_{it} - E[\bar{\err}_{it}\right] )^2]\right\}~\left\|\sum_{i=1}^N\E\left[\sum_{t=1}^{t_i}e_{i}(N)e_{i}(N)^\top\right]\right\|_\op\\
			&\le 2\sigma^2 T
		\end{align}
		and
		\begin{align}
			\left\|\sum_{i=1}^N\E[\bz_i^\top\bz_i]\right\|_\op &\le 2\sigma^2\left\|\sum_{i=1}^N\E\left[\sum_{t=1}^{t_i}e_{t}(T)e_{t}(T)^\top\right]\right\|_\op\\
			&= 2\sigma^2N\,.
		\end{align}
		Note that here we used the fact that random variables $\bar{\err}_{it} - E[\bar{\err}_{it}]$ are centered, $\tau^2$-sub-Gaussian, and independent of each other (as $i$ or $t$ vary), which means all cross terms of the type $\E\{(\bar{\err}_{it} - E[\bar{\err}_{it}])(\bar{\err}_{js} - E[\bar{\err}_{js}])\}$ are zero for $(i,t)\neq(j,s)$. Therefore, $\sigma_Z^2 = 2\sigma^2\max(N,T)$ works. Applying Proposition \ref{prop:matrix_bernstein}, we obtain
		\begin{align*}
			\prob\left\{
			\left\|\sum_{i=1}^N\bz_i\right\|_\op \ge \alpha
			\right\}
			&\le (N+T)\exp\left[-\frac{\alpha^2}{4\sigma^2\max(N,T)+(4D\alpha\sqrt{T})/3}\right]\\
			&\\
			&\le (N+T)\exp\left[-\frac{3}{16}\min\left(\frac{\alpha^2}{\sigma^2 \max(N,T)},\frac{\alpha}{D\sqrt{T})}\right)\right]\,.
		\end{align*}
		Therefore, there is a constant $C_3$ such that with probability greater than $1-\exp(-t)$,
		\begin{align*}
			\left\|\sum_{i=1}^N\bz_i\right\|_\op \le C_3\sigma\max\left(\sqrt{\max(N,T)[t+\log(N+T)]},\sqrt{T\log(N+T)}[t+\log(N+T)]\right)\,.
		\end{align*}
		Using this for a $t$ that is a large enough constant times $\log(N+T)$, together with \eqref{eq:P(E)<P(sum_B)+err} and \eqref{eq:sumB<sumZ+err}, shows with probability larger than $1-2(N+T)^{-3}$
		\begin{align*}
			\left\|\be\right\|_\op &\le C_1\sigma\max\left[\sqrt{\max(N,T)\log(N+T)},\sqrt{T}\log^{3/2}(N+T)\right]\\
			&= C_1\sigma\max\left[\sqrt{N\log(N+T)},\sqrt{T}\log^{3/2}(N+T)\right]\,,
		\end{align*}
		for a constant $C_1$.
	\end{proof}
	
\subsection{Proof of Lemma \ref{lem:lower_bound_err}}\label{subsec:pf-lem-lower-bound-err}

Proof of Lemma \ref{lem:lower_bound_err} is similar to the proof of Theorem 1 in \citep{negahban2012restricted},  Lemma 12 in \citep{klopp2014noisy}, or Corollary 3.1 in \citep{hamidi2019lowrank}. However, for completeness, below we provide a summary of this proof (adapted to our setting).

Recall that our aim is to prove that when $\hat\bl$ satisfies $\|\hat\bl-\bl^*\|_F^2\ge4\lmax^2\theta/\pc$ for a positive number $\theta$, then for constant $C\ge 0.001$, when $C\theta>T$, we have
\[
\prob_\pi\left\{\frac{\pc}{2}\,\|\hat\bl-\bl^*\|_F^2 > \sum_{(i,t)\in\calo}\langle\ba_{it},\hat\bl-\bl^*\rangle^2+8\lmax^2T\sqrt{N}\right\}
\le 2\exp\left(-\frac{C\theta}{T}\right)\,.
\]
First, we define some additional notation. Given the observation set $\calo$, for every $N$ by $T$ matrix $\bM$ define $\cX_\calo(M)$ to be an $|O|$ by $1$ vector that is obtained by stacking observed entries of all rows of $\bM$ vertically. Specifically, for all $i\in[N]$ we define
\[
\cX_\calo^{(i)}(\bM)\equiv[\langle\ba_{i1},\bM \rangle,\ldots,\langle\ba_{it_i},\bM \rangle]^\top\,,
\]
and then define,
\[
\cX_\calo(\bM)\equiv
\left[\begin{array}{c}\cX_\calo^{(1)}(\bM)\\
\vdots\\
\cX_\calo^{(N)}(\bM)
\end{array}
\right]\,.
\]
In addition, we define $\ltpi$ norm of $\bM$ by
\[
\|\bM\|_{\ltpi}\equiv\sqrt{\E_\pi\Big(\|\cX_\calo(\bM)\|_2^2\Big)}\,.
\]
We also define the constraint set
\[
\calc(\theta)\equiv\left\{\bM\in\reals^{N\times T}~|~\|\bM\|_{\max}\le1\,,~\|\bM\|_{\ltpi}^2\ge \theta\right\}\,.
\]
Now we are ready to prove Lemma \ref{lem:lower_bound_err}.
\begin{proof}[Proof of Lemma \ref{lem:lower_bound_err}]
First, define $\vartheta=T\sqrt{N}$. We show that proof of Lemma \ref{lem:lower_bound_err} would be a corollary of the following statement. Whenever $\bM\in\cC(\theta)$, there is a constant $C_4\ge 0.001$ such that, when $C_4\theta>T$, we have
\begin{align}
\prob_\pi\left\{ \frac{1}{2}\|\bM\|_{\ltpi}^2>\|\cX_\calo(\bM)\|_2^2+2\vartheta\right\}\le 2\exp\left(-\frac{C_4\theta}{T}\right)\,.\label{eq:lemma3-l2pi-version}
\end{align}
The reason for the fact that Lemma \ref{lem:lower_bound_err} follows from this statement is as follows. It is straightforward to see that for all $\bM$, $\|\bM\|_{\ltpi}^2\ge \pc\|\bM\|_F^2$ which means that the assumption $\|\hat\bl-\bl^*\|_F^2\ge4\lmax^2\theta/\pc$ results in $\|\hat\bl-\bl^*\|_{\ltpi}^2\ge4\lmax^2\theta$.  Therefore, defining $\bM=(2\lmax)^{-1}(\hat\bl-\bl^*)$, it is clear that $\bM$ is in $\cC(\theta)$. We can now apply Eq. \eqref{eq:lemma3-l2pi-version} to this choice of $\bM$ and obtain
\begin{align*}
\prob_\pi\left\{\frac{\pc}{2}\,\|\hat\bl-\bl^*\|_F^2 > \sum_{(i,t)\in\calo}\langle\ba_{it},\hat\bl-\bl^*\rangle^2+8\lmax^2T\sqrt{N}\right\}
&= \prob_\pi\left\{\frac{\pc}{2}\,\|\bM\|_F^2 > \|\cX_\calo(\bM)\|_2^2+2\vartheta\right\}\\
&\le \prob_\pi\left\{\frac{\|\bM\|_{\ltpi}^2}{2} > \|\cX_\calo(\bM)\|_2^2+2\vartheta\right\}\\
&\le
2\exp\left(-\frac{C_4\theta}{T}\right)\,,
%
\end{align*}
which is what we needed.

Therefore, in the remaining of the proof of Lemma \ref{lem:lower_bound_err}, we will prove that Eq. \eqref{eq:lemma3-l2pi-version} holds. Let us define the following bad event,
\[
\cB\equiv\left\{
\exists \bM\in\cC(\theta)~\Big|~ \|\bM\|_{\ltpi}^2-\|\cX_\calo(M)\|_2^2\ge \frac{1}{2}\|\bM\|_{\ltpi}^2+2\vartheta
\right\}\,.
\]
Our goal will be to bound probability of the event $\cB$. Let also $\xi=2$, and define for every $\rho\ge \theta$,
\[
\calc(\theta,\rho)\equiv\left\{\bM\in\cC(\theta)~\Big|~\rho\le \|\bM\|_{\ltpi}^2\le \rho\xi\right\}\,.
\]
Therefore, $\cC(\theta)=\cup_{\ell=1}^\infty \calc(\theta,\theta\xi^{\ell-1}) $. Now, assume that $\bM\in\calc(\theta)$ such that event $\cB$ holds for $\bM$. This means for some $\ell\ge1$, $\bM\in\calc(\theta,\xi^{\ell-1}\theta)$ and event $\cB$ holds for $\bM$. Thus,
\begin{align*}
\|\bM\|_{\ltpi}^2-\|\cX_\calo(\bM)\|_2^2&\ge \frac{1}{2}\|\bM\|_{\ltpi}^2+2\vartheta\\
&\ge \frac{1}{2\xi}\xi^\ell\theta+2\vartheta\,.
\end{align*}
Now, define the event $\cB_\ell$ by,
\[\cB_\ell\equiv\left\{\exists \bM\in\cC(\theta,\xi^{\ell-1}\theta)~\Big|~
\|\bM\|_{\ltpi}^2-\|\cX_\calo(\bM)\|_2^2\ge \frac{1}{2\xi}\xi^\ell\theta+2\vartheta
\right\}\,.
\]
This means $\cB\subseteq \cup_{\ell=1}^\infty\cB_\ell$.

Our next step is to use the following concentration inequality that is proved at the end of this section.
\begin{lemma}\label{lem:massart-extension}
Let $Z_{\rho}\equiv\sup_{\bM\in\cC(\theta,\rho)}\left\{ \|\bM\|_{\ltpi}^2-\|\cX_\calo(\bM)\|_2^2\right\}$, then there exist a constant $C_3\ge 0.0025$ such that
\begin{align}
\prob_\pi\left\{Z_\rho\ge \frac{1}{2\xi}\xi\rho+2\vartheta\right\}\leq \exp\left(-\frac{C_3\rho\xi}{T}\right)\,.\label{eq:Massart-tail}
\end{align}
\end{lemma}
Assuming Lemma \ref{lem:massart-extension}, we have
\begin{align*}
\prob_\pi(\cB_\ell)&\leq \exp\left(-\frac{C_3\xi^\ell\theta}{T}\right)\\
&\leq \exp\left(-\frac{C_3\ell\log(\xi)\theta}{T}\right)\,,
\end{align*}
where the last step uses $x\ge\log(x)$ for $x>1$. Therefore, we can choose
\[
C_4=C_3\log(\xi)\ge 0.0025\times\log(2)\ge 0.001\,,
\]
and apply the union bound to obtain
\begin{align*}
\prob_\pi(\cB)&\leq\sum_{\ell=1}^\infty \exp\left(-\frac{C_4\ell\theta}{T}\right)\\
&=\frac{\exp\left(-\frac{C_4\theta}{T}\right)}
{1-\exp\left(-\frac{C_4\theta}{T}\right)}\,.
\end{align*}
If $C_4\theta>T$,  we obtain
\begin{align*}
\prob_\pi(\cB)&\leq 2\exp\left(-\frac{C_4\theta}{T}\right)\,,
\end{align*}
which finishes the proof of Eq. \eqref{eq:lemma3-l2pi-version}, and hence, completes the proof of Lemma \ref{lem:lower_bound_err}.
\end{proof}
\begin{proof}[Proof of Lemma \ref{lem:massart-extension}]
Let
\[\tZ_{\rho}\equiv\sup_{\bM\in\cC(\theta,\rho)}\left\{\Big| \|\bM\|_{\ltpi}^2-\|\cX_\calo(\bM)\|_2^2\Big|\right\}\,.
\]
Clearly $Z_\rho\le \tZ_\rho$ therefore, if we prove expression \eqref{eq:Massart-tail} holds for $\tZ_\rho$, then it would hold for $Z_\rho$ as well. We aim to use Massart's concentration inequality (Theorem 3 of \cite{massart2000constants}). But, first we would need to find upper bounds for $\E(\tZ_\rho)$ and a certain variance term, $\sigma_{\text{Massart}}^2$. Specifically,
\begin{align*}
\sigma_{\text{Massart}}^2:=\sup_{\bM\in\cC(\theta,\rho)}\sum_{i=1}^{N}\text{Var}\left[
\|\cX_\calo^{(i)}(\bM)\|_2^2\right]&\le \sup_{\bM\in\cC(\theta,\rho)}\sum_{i=1}^{N}\E\left[ \|\cX_\calo^{(i)}(\bM)\|_2^4\right]\\
&\le T\sup_{\bM\in\cC(\theta,\rho)}\sum_{i=1}^{N}\E\left[ \|\cX_\calo^{(i)}(\bM)\|_2^2\right]\\
&= T\sup_{\bM\in\cC(\theta,\rho)}\|\bM\|_{\ltpi}^2\\
&\le T\rho\xi\,,
\end{align*}
where the second and the last inequality use definition of $\cC(\theta,\rho)$ that provides upper bounds for $\|\bM\|_{\ltpi}$ and $\|\bM\|_{\max}$.

Next, we work on finding an upper bound for $\E(\tZ_\rho)$. First, using a symmetrization argument (Lemma 6.3 of \cite{ledoux2013probability}), we have
\begin{align*}
\E(\tZ_\rho)&\leq 2 \E\left[\sup_{\bM\in\cC(\theta,\rho)} \Big|\sum_{i=1}^{N}\zeta_i\|\cX_\calo^{(i)}(\bM)\|_2^2\Big|\right]\,,
\end{align*}
where $(\zeta_i)_{i=1}^n$ are iid Rademacher random variables. Note that we used identity function for $F$ and norm-infinity on an infinite dimensional vector indexed by matrices $\bM$.
Now, we have,
\begin{align*}
\E\left[\sup_{\bM\in\cC(\theta,\rho)} \Big|\sum_{i=1}^{N}\zeta_i\|\cX_\calo^{(i)}(\bM)\|_2^2\Big|\right]&\le
T\,\E\left[\Big|\sum_{i=1}^{N}\zeta_i\Big|\right]\\
&\le T\,\sqrt{\E\left[\Big(\sum_{i=1}^{N}\zeta_i\Big)^2\right]}\\
&=T\sqrt{N}\,.
\end{align*}
Here, the first inequality uses the bound on $\|\bM\|_{\max}$ in definition of $\cC(\theta,\rho)$, and the second inequality relies on the Jensen's inequality.

As our final step, and recalling that each term $\|\cX_\calo^{(i)}(\bM)\|_2^2$ is at most $T$, we can invoke Massart's inequality with $b_{\text{Massart}}=T$, $x_{\text{Massart}}=C_3\rho\xi/T$ for a small enough constant $C_3$ that is bigger than $0.0025$, and $\varepsilon_{\text{Massart}}=1$, to obtain
\begin{align*}
\prob_\pi\left\{\tZ_\rho\ge \frac{1}{2\xi}\rho\xi + 2\vartheta\right\}&\le
\prob_\pi\left\{\tZ_\rho\ge 2\vartheta+(\sqrt{8C_3}+35C_3)\rho\xi \right\}\\
&\le\prob_\pi\left\{\tZ_\rho\ge 2\E[\tZ_\rho]+\sigma_{\text{Massart}}\sqrt{8x_{\text{Massart}}}+34.5\,b_{\text{Massart}}x_{\text{Massart}}\right\}\\
&\le e^{-x_{\text{Massart}}}=\exp\left(-\frac{C_3\rho\xi}{T}\right)\,,
\end{align*}
which finishes the proof of Lemma \ref{lem:massart-extension}.
\end{proof}

\end{document}